\documentclass[11pt]{article}

\usepackage{amsmath,amsthm,amscd,latexsym}
\usepackage{amsfonts,pdfsync,color,graphicx}
\usepackage[psamsfonts]{amssymb}
\usepackage{epsfig}
\usepackage{color}
\usepackage[noadjust]{cite}

\usepackage{listings}

\usepackage[top=1.5in,bottom=1.5in,left=1.3in,right=1.0in]{geometry}

\usepackage{tikz}

\newcommand{\R}{{{\mathbb R}}}

\newtheoremstyle{mystyle}               % Name
  {}                % Space above
  {}                % Space below
  {}        % Body font
  {}                % Indent amount
  {\bfseries \itshape}       % Theorem head font
  {.}      % Punctuation after theorem head
  { }      % Space after theorem head, ' ', or \newline
  {}       % Theorem head spec

\newtheorem{theorem}{Theorem}[section]
\newtheorem{proposition}[theorem]{Proposition}
\newtheorem{lemma}[theorem]{Lemma} 
\newtheorem{corollary}[theorem]{Corollary}
\theoremstyle{definition}

\newtheorem{example}[theorem]{Example}	
\theoremstyle{mystyle}
\newtheorem{remark}[theorem]{Remark}

\title{Green's Functions and Spectral Theory for the Hill's Equation \footnote{A. Cabada  was partially supported by Ministerio de Educaci\'on y Ciencia, Spain, and FEDER, Project MTM2013-43014-P. J. A. Cid was partially supported by Xunta de Galicia (Spain), project EM2014/032.}}
\date{}
\author{Alberto Cabada$^1$, Jos\'e A. Cid$^2$ and Luc{\' i}a L\'opez Somoza$^1$\\
$^1$ Departamento de An\'alise Ma\-te\-m\'a\-ti\-ca, Facultade de Matem\'aticas, \\
Universidade de Santiago de Com\-pos\-te\-la, 15782 Santiago de Compostela,\\ Galicia, Spain.\\
alberto.cabada@usc.es; lucia.lopez.somoza@gmail.com\\
$^2$ Departamento de Matem\'aticas, Universidade de Vigo, 32004, Pabell\'on 3,\\ Campus de Ourense, Galicia, Spain.\\
angelcid@uvigo.es}

\begin{document}
\maketitle
\begin{abstract}
The aim of this paper is to show certain properties of the Green's functions related to the Hill's equation coupled with different two point boundary value conditions. We will obtain the expression of the Green's function of Neumann, Dirichlet, Mixed and anti-periodic problems as a combination of the Green's function related to periodic ones. 

As a consequence we will prove suitable results in spectral theory and deduce some comparison results for the solutions of the Hill's equation with different boundary value conditions.
\end{abstract}



\section{Introduction}
The Hill's equation, 
\begin{equation} \label{e-hill-intro}\tag{1}
u''+a(t)\,u=0,
\end{equation}
has numerous applications in engineering and physics. We can find, among others, some problems in mechanics, astronomy, circuits, electric conductivity of metals and cyclotrons. 

Moreover, the theory related to the Hill's equation can be extended to every differential equation written in the general form
\begin{equation}\label{e-homog-segundo-orden}\tag{2}
u''+p(t)\,u'+q(t)\,u=0
\end{equation}
such that the coefficients $p$ and $q$ have enough regularity. This is due to the fact that, with a suitable change of variable, the previous equation transforms in one of the type of (\ref{e-hill-intro}) (Details can be seen in \cite{simmons}).

As a first example  let us consider a mass-spring system, that is, a spring with a mass $m$ hanging from it. Denoting by $x(t)$ the position of the mass at the instant $t$, by considering a time variable friction coefficient $\mu(t)$ (for instance, if the spring moves between two different environments, with different coefficient of friction in each of them, or in an environment where there are important variations of density or temperature that cause changes in the friction coefficient depending on the instant of the  considered process) and an external force $F(t)$ acting periodically on the mass in such a way that it tends to move the mass back into its position of equilibrium, acting in proportion to the distance to that position, we obtain the model 
$$x''(t)+\mu(t)\,x'(t)+\left(\frac{k}{m}+F(t)\right)x(t)=0.$$

Thus, we arrive to an equation in the form (\ref{e-homog-segundo-orden}) in which, if $\mu(t)$ has enough regularity, we could transform the equation in one in the form (\ref{e-hill-intro}), where the potential $a(t)$ follows the expression (see \cite{simmons} for details)
$$a(t)=\frac{k}{m}+F(t)-\frac{1}{4}\left(\frac{\mu(t)}{m}\right)^2-\frac{1}{2}\,\frac{\mu'(t)}{m}.$$

We note that in such a case, even if the data involved in the equation ($\mu$ and $F$) have constant sign on its interval of definition, such property may be not true for the potential $a$.

A second example studied in \cite{csizmadia} is a mathematical (or inverted) pendulum. If we assume that the oscillations of the pendulum are small and that the suspension point of the string vibrates vertically with an acceleration $a(t)$ then, as it is proved in \cite{csizmadia}, the movement would be modeled by the equation (which follows the form (\ref{e-hill-intro}))
$$\theta''(t)-\frac{1}{l}\,(g+a(t))\,\theta(t)=0,$$
where $g$ denotes the gravity, $l$ the length of the string and $\theta$ represents the angle between the string and the perpendicular line to the base. 

Other equations that fit on the framework of Hill's equation are Airy's equation, $u''(t)+t\,u(t)=0$ (see \cite{simmons}), and Mathieu's equation, $u''(t)+(c+b\,\cos{t})\,u(t)=0$ (see \cite{zhanli,cacid1,torres1}).

At the moment of studying oscillation phenomena of the solutions of \eqref{e-hill-intro}, it is observed that these are determined by the potential $a(t)$. In particular, solutions of \eqref{e-hill-intro} do not oscillate when $a(t)<0$ but they do it infinite times for $a(t)>0$ large enough.
Moreover, the larger the potential $a(t)$ is, the faster the solutions of \eqref{e-hill-intro} oscillate.

By simply considering that every integrable function can be rewritten as
$$a(t)=\frac{1}{T}\int_{0}^{T}a(s)\,ds +\tilde{a}(t),$$
it is obvious that studying the potentials $a(t)$ for which the solutions of the equation (\ref{e-hill-intro}) oscillate in $[0,T]$, is equivalent to study the values of $\lambda\in \mathbb{R}$ for which the equation
\begin{equation} \label{e-hill-intro2}\tag{3}
u''+[a(t)+\lambda]\,u=0, \quad t\in[0,T],
\end{equation}
with $a\in L^{\alpha}[0,T]$ fixed, $\alpha\ge 1$, has no trivial solution.

If we consider the equation (\ref{e-hill-intro2}) coupled with suitable boundary value conditions, we have a spectral problem.

First studies about the Hill's equation are focused on the homogeneous case, from the point of view of the classical oscillation theory of Sturm-Liouville (\cite{simmons,magnus}). In particular, from the study of the equation (\ref{e-hill-intro2}) under periodic boundary conditions, important results related to the stability of solutions were obtained.

Afterwards the non homogeneous periodic problem,
$$u''(t)+[a(t)+\lambda]\,u(t)=\sigma(t), \; t\in[0,T], \quad
u(0)=u(T), \quad u'(0)=u'(T),$$
was studied, with $a\in L^{\alpha}[0,T]$, $\alpha\ge 1$ and $\sigma\in L^1[0,T]$. In this case it results specially interesting the study of constant sign solutions when $\sigma$ does not change sign. This situation could be interpreted in a physical way by considering $\sigma$ as an external force acting over the system; then constant sign solutions would mean that a positive perturbation maintains oscillations above or below the equilibrium point.

The study of constant sign solutions carried to consider comparison principles (that is, maximum and antimaximum principles) which, later, were related to the constant sign of the Green's function.

The non homogeneous periodic problem has ben widely considered in the literature, see \cite{cacid, cacid1, cacidtv, torres1, zhangMN05, zhang, zhanli} and references therein.

Nevertheless, it is also interesting the study of different boundary conditions that frequently appear in the field of differential equations as, for instance, Neumann, Dirichlet, Mixed and anti-periodic.

The main purpose of this paper, given in Section 3,  consists on obtaining the explicit expression of the Green's function related to Neumann, Dirichlet, Mixed and anti-periodic boundary conditions, as a linear combination of the Green's function related to the periodic problem. As a flavor of our results see for instance formulas \eqref{e-Green-NP} and \eqref{e-Green-DP}. The importance of these results resides in the fact that the Green's function of a non homogeneous problem completely characterizes its solutions. In particular, this paper proves that the study of all boundary value problems for Hill's equation could be reduced to the study of the Green's function of the periodic one, which, as we have noted above, has been widely treated in the recent literature. It is important to mention that at the beginning of that section we prove a general result satisfied by the Green's function of a general self-adjoint operator.

Once we have such expressions, in order to assure the constant sign of the Green's function related to the different types of boundary value problems considered in Section 3,  we apply in Section 4 previous results given in the literature for the periodic problem. This way, we obtain conditions to warrant the existence of constant sign solutions for Neumann, Dirichlet and Mixed problems without doing a direct study of such problems. Moreover, we are able to compare their constant sign. As consequence, we deduce direct relations between the Green's functions. These results allow us to obtain comparison principles which warrant that, for certain intervals of the parameter $\lambda$, the solution of the non homogeneous Hill's equation under some conditions is bigger at every point than the solution of the same equation under another type of boundary conditions.

We also obtain a decomposition of the spectrum of some problems as a combination of the other ones; this allows us to deduce a certain order of appearance for the eigenvalues of each problem. Moreover, we include some numerical examples that hint an order of eigenvalues even more precise than the one theoretically proved.

In order to do the paper self-contained we start, in next section, with a summary of the main known results related to the periodic problem.

\section{Preliminaries}
We introduce now the notation and definitions we will use all along the paper.

\vspace*{0.1cm}
Let $L[a]$ be the Hill's operator associated to potential $a$
\begin{equation}
\label{e-L[a]}
L[a]\,u(t)\equiv u^{\prime\prime}(t)+a(t)\,u(t), \quad t \in [0,T],
\end{equation}
with
$a: I \to \R, \ a\in L^\alpha(I),\; \alpha\geq 1,$
where $I=[0,T]$. We will denote by $J=[0,2 \,T]$. 

\vspace*{0.1cm}
We will work both with the positive part ($a_+(t)=\max\{a(t),\,0\}, \  t\in I$) and with the negative one ($a_-(t)=-\min\{a(t),\,0\}, \ t\in I$) of potential $a$.

\vspace*{0.1cm}
On the other hand, given $1\le \alpha \le \infty$ we denote by $\alpha^*$ its conjugate, that is, the number satisfying the relation $\displaystyle \frac{1}{\alpha}+\frac{1}{\alpha^*}=1$  (with $\alpha=1$ and $\alpha^*=\infty$ and vice-versa). 

\vspace*{0.1cm}
Moreover, $h\succ 0$ means a function $h\in L^{\alpha}(c,d)$ such that $h(t)\ge 0$ for a.\,e. $t\in [c,d]$ and $h\not\equiv 0$ on $[c,d]$, where a.\,e. means in every point except for a set with measure zero. 

\vspace*{0.1cm}
We will work with the space $W^{2,1}(c,d)$, defined as the set of functions $u\in C^1[c,d]$ such that $u'$ is absolutely continuous on $[c,d]$ (in the sequel $AC[c,d]$), with $c<d$ fixed in each case.

\vspace*{0.1cm}
Let then $X\subset W^{2,1}(I)$ be a Banach space such that the homogeneous equation 
\begin{equation*}
L[a]\,u(t)=0\quad \mbox{ a. e. }\; t\in I,\qquad u\in X,
\end{equation*}
has only the trivial solution.

Previous condition is known as {\it operator $L[a]$ is nonresonant in $X$}. It is very well known that if $\sigma \in L^1(I)$ and this condition is satisfied, we have that problem
$$L[a]\,u(t)=\sigma(t)\quad \mbox{ a. e.  }\;t\in I,\qquad u\in X,$$
has a unique solution given by
$$u(t)=\int_0^T G[a,T](t,s)\,\sigma(s)\,ds,\qquad \forall\; t\in I.$$
$G[a,T](t,s)$ is the so-called Green's function related to operator $L[a]$ in $X$ and it is uniquely determined. See \cite{Cab} for details.
\vspace*{0.1cm}

We say that $L[a]$ admits the maximum principle (MP) in $X$ if and only if
$$u\in X, \quad L[a]\, u \ge 0 \ \mbox{on $I$}\Longrightarrow u \le  0 \ \mbox{on $I$},$$
and $L[a]$ admits the antimaximum principle (AMP) in $X$ if and only if
$$u\in X, \quad L[a]\, u \ge 0 \ \mbox{on $I$}\Longrightarrow u\ \ge 0 \ \mbox{on $I$}.$$

It is immediate to verify that if $L[a]$ satisfies MP or AMP in $X$ then it is nonresonant in $X$.

It is very well known (\cite{Cab,Cop,magnus}) that the operator $L[a]$ is non resonant and self-adjoint on $X$ if and only if the related Green's function exists and is symmetrical with respect to the diagonal of its square of definition, that is,
\[
G[a,T](t,s)=G[a,T](s,t),\qquad \forall \;(t,s)\in I \times I.
\]

In the sequel we compile several known results related to the Green's function of the periodic boundary value problem
\[
\label{e-P} L[a]\,u(t)=0\quad \mbox{ a. e.  }\; t\in I,\quad
u(0)=u(T),\; u'(0)=u'(T),
\]
which we will denote as $G_P[a,T](t,s)$.

\begin{lemma}\label{l-per}
\cite[Lemma 2.2]{cacid}
Suppose that the Green's function $G_P[a,T]$ does not change sign on $I\times I$ and vanishes at some point $(t_0,s_0)\in I\times I$, then $t_0=s_0$, $(t_0,s_0)=(0,T)$ or $(t_0,s_0)=(T,0)$.
\end{lemma}

\begin{lemma}
\label{l-Green-comparison} 
\cite[Lemma 2.3]{cacid}
The following claims are equivalent:
\begin{itemize}

\item[(1)] $G_P[a,T](t,s) \ge 0$ $(\le 0)$ on $I \times I$.

\item[(2)] If $u \in W^{2,1}(I)$, $u(0)=u(T),\; u'(0)=u'(T)$, and $L[a] \, u \succ 0$ on $I$ then $u>0$ $(<0)$ on $I$.

\end{itemize}
\end{lemma}

The previous lemma establishes the equivalence between the constant sign of the Green's function and the strict maximum and antimaximum principles. 

\begin{lemma}
\cite[Lemma 2.4]{cacid}
\label{l-Green-neg}
If $G_P[a,T] \le 0$  on $I \times I$ then $G_P[a,T]< 0$ on $I \times I$.
\end{lemma}

\begin{lemma}
\cite[Lemma 2.8]{cacid}
\label{l-Green-dec}
Let $a_1,\, a_2 \in L^\alpha(I)$ be such that the Green's functions of the corresponding problem, $G[a_1,T]$ and $G[a_2,T]$, have the same constant sign on $I \times I$. If $a_1 \succ a_2$ on $I$ then $G[a_1,T](t,s) < G[a_2,T](t,s)$ for all $(t,s) \in I \times I$.
\end{lemma}

\begin{lemma}
\label{l-Green-int}
\cite[Corollary 3.1]{cacid}
Let $a_1,\, a_2 \in L^\alpha(I)$ be such that $a_1 \succ a_2$ on $I$ and assume that the Green's functions $G[a_1,T]$ and $G[a_2,T]$ have the same constant sign on $I \times I$. Then, for all $a \in L^\alpha(I)$ with $a(t) \in [a_2(t),a_1(t)]$ for a.\,e. $t \in I$, $G[a,T]$ exists and it has the same constant sign as $G[a_1,T]$ and $G[a_2,T]$.
\end{lemma}

Let ${\lambda}_P(a,T)$ be the smallest eigenvalue of the periodic equation
$$
u''(t)+(a(t) + \lambda)\, u(t)=0,\quad \mbox{ a. e.  }\;t\in I, \qquad u(0)=u(T),\; u'(0)=u'(T),
$$
and let ${\lambda}_A(a,T)$ be the smallest eigenvalue of the anti-periodic equation
$$
u''(t)+(a(t) + \lambda)\, u(t)=0,\quad \mbox{ a. e.  }\;t\in I,
\qquad u(0)=-u(T),\; u'(0)=-u'(T).
$$

\begin{lemma} \cite[Theorem 1.1]{zhang}
\label{l-zhang-eigen}
Suppose that $a \in L^1(I)$, then:
\begin{enumerate}
\item $G_P[a,T](t,s) \le 0$ on $I \times I$ if and only if ${\lambda}_P(a,T)>0$.
\item $G_P[a,T](t,s) \ge 0$ on $I \times I$ if and only if ${\lambda}_P(a,T)< 0\le {\lambda}_A(a,T)$.
\end{enumerate}
\end{lemma}

By introducing the parametrized potentials $a+\lambda$, with $\lambda \in \R$, the previous result could be rewritten as follows
\begin{lemma} \cite[Theorem 1.2]{zhang}
\label{l-zhang-eigen-2}
Suppose that $a \in L^1(I)$, then:
\begin{enumerate}
\item $G_P[a+\lambda,T](t,s) \le 0$ on $I \times I$ if and only if $\lambda < {\lambda}_P(a,T)$.
\item $G_P[a+\lambda,T](t,s) \ge 0$ on $I \times I$ if and only if ${\lambda}_P(a,T)< \lambda \le {\lambda}_A(a,T)$.
\end{enumerate}
\end{lemma}

On the other hand, the following Oscillation Theorem establishes a certain relation of order between the eigenvalues of the equation
\begin{equation}\label{e-hill}
u''(t)+(a(t)+\lambda)\,u(t)=0
\end{equation}
associated to periodic and anti-periodic problems.

\begin{theorem}[Oscillation] \cite[Chapter 2]{magnus} \label{t-oscilacion} Denote
$$\lambda_0, \,\lambda_1, \, \lambda_2 \dots \quad \text{and} \quad \lambda'_1,\, \lambda'_2,\, \lambda'_3 \dots$$
as the eigenvalues of \eqref{e-hill-intro} associated to periodic and anti-periodic boundary conditions, respectively. Then
$$\lambda_0<\lambda'_1\le \lambda'_2<\lambda_1 \le \lambda_2<\lambda'_3\le \lambda'_4<\lambda_3\le \lambda_4 \dots$$

Moreover, the eigenvalues $\lambda_n$ are characterized as the infinite roots of the equation $\Delta(\lambda)=2$ and the eigenvalues $\lambda'_n$ as the roots of $\Delta(\lambda)=-2$, where
$$\Delta(\lambda)=y_1(T,\lambda)+y'_2(T,\lambda)$$
and $y_1$, $y_2$ are the solutions of \eqref{e-hill} satisfying the initial conditions
$$y_1(0,\lambda)=1, \ y'_1(0,\lambda)=0,$$
$$y_2(0,\lambda)=0, \ y'_2(0,\lambda)=1.$$

The trivial solution of \eqref{e-hill-intro} is stable if and only if $\lambda$ belongs to the intervals
$$(\lambda_0, \lambda'_1),\, (\lambda'_2, \lambda_1),\, (\lambda_2, \lambda'_3),\, (\lambda'_4, \lambda_3), \dots$$
\end{theorem}

Graphically, the function $\Delta(\lambda)$ has the following appearance
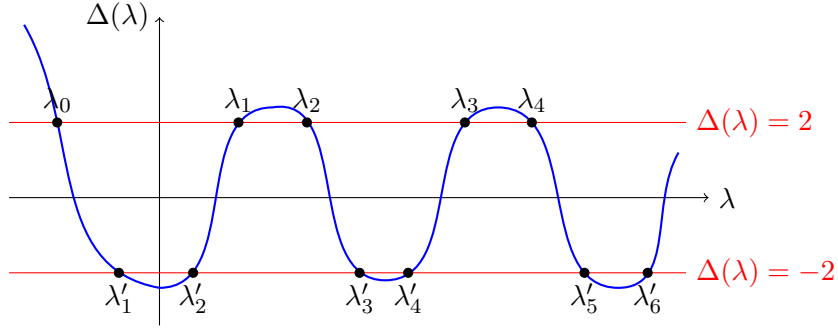
\begin{center}
\begin{tikzpicture}
\draw[->] (-2,0) -- (7.3,0) node[right] {$\lambda$};
\draw[->] (0,-1.7) -- (0,2.4) node[left] {$\Delta(\lambda)$};
\draw[red] (-2,1) -- (7,1) node[right] {$\Delta(\lambda)=2$};
\draw[red] (-2,-1) -- (7,-1) node[right] {$\Delta(\lambda)=-2$};
\draw [blue, thick] (-1.8,2.3) to [out=-60,in=170] (0,-1.2) to [out=0,in=180] (1.5,1.2) to [out=10,in=180] (3,-1.1) to [out=0,in=180] (4.5,1.2) to [out=0,in=180] (6.1,-1.2) to [out=0,in=240] (6.9,0.6);
\draw[fill] (-1.36,1) circle (0.06cm) node[above] {$\lambda_0$};
\draw[fill] (1.05,1) circle (0.06cm) node[above] {$\lambda_1$};
\draw[fill] (1.96,1) circle (0.06cm) node[above] {$\lambda_2$};
\draw[fill] (4.06,1) circle (0.06cm) node[above] {$\lambda_3$};
\draw[fill] (4.95,1) circle (0.06cm) node[above] {$\lambda_4$};
\draw[fill] (-0.54,-1) circle (0.06cm) node[below] {$\lambda'_1$};
\draw[fill] (0.447,-1) circle (0.06cm) node[below] {$\lambda'_2$};
\draw[fill] (2.66,-1) circle (0.06cm) node[below] {$\lambda'_3$};
\draw[fill] (3.306,-1) circle (0.06cm) node[below] {$\lambda'_4$};
\draw[fill] (5.65,-1) circle (0.06cm) node[below] {$\lambda'_5$};
\draw[fill] (6.49,-1) circle (0.06cm) node[below] {$\lambda'_6$};
\end{tikzpicture}
\end{center}

On the other hand, denoting ${\lambda}_D(a,T)$ as the smallest eigenvalue of the Dirichlet problem
$$
u''(t)+(a(t) + \lambda)\, u(t)=0,\quad \mbox{ a. e. }\;t\in I,
\quad u(0)=u(T)=0,
$$
and $G_D[a+\lambda,T]$ as the corresponding Green's function, we have the following very well known result

\begin{lemma}\label{l-coppel} 
Suppose that $a \in L^1(I)$, then:

 $G_D[a+\lambda,T](t,s)<0$ for all $(t,s) \in (0,T) \times (0,T)$ if and only if $\lambda< \lambda_D(a,T)$.
\end{lemma}

We finalize this preliminary section by showing two particular cases of some more general spectral results given in \cite[Lemmas 1.8.25 and 1.8.33]{Cab}.

\begin{lemma}\label{l-M-NT}
Suppose that operator $L[a]$ is nonresonant in the Banach space $X$, its related Green's function $G[a,T]$ is nonpositive on $I \times I$, and satisfies condition
\begin{enumerate}
\item[$(N_g)$] There is a continuous function $\phi(t) >0$ for all $t \in (0,T)$ and $k_1, \; k_2 \in L^1(I)$, such that $k_1(s)<k_2(s)<0$ for a.e. $s \in I$, satisfying
$$ \phi(t)\, k_1(s) \le  G[a,T](t,s) \le \phi(t)\, k_2(s), \quad \mbox{for a.e. } (t,s)  \in I \times I.$$
\end{enumerate}

Then $G[a+\lambda,T]$ is nonpositive on $I \times I$ if and only if $\lambda\in (-\infty, \lambda_1(a,T))$ or  $\lambda\in [-\bar{\mu}(a,T), \lambda_1(a,T))$, with $\lambda_1(a,T) >0$ the first eigenvalue of operator $L[a]$ in $X$ and ${\bar{\mu}(a,T) \ge 0}$ such that $L[a -\bar{\mu}]$ is nonresonant in $X$ and the related nonpositive Green's function $G[a -\bar{\mu},T]$ vanishes at some point of the square $I\times I$.
\end{lemma}

\begin{lemma}
\label{l-M-PT}
Suppose that operator $L[a]$ is nonresonant in the Banach space $X$, its related Green's function $G[a,T]$ is nonnegative on $I \times I$, and satisfies condition
\begin{enumerate}
\item[$(P_g)$] There is a continuous function $\phi(t) >0$ for all $t \in (0,T)$ and $k_1, \; k_2 \in L^1(I)$, such that $0<k_1(s)<k_2(s)$ for a.e. $s \in I$, satisfying
$$ \phi(t)\, k_1(s) \le  G[a,T](t,s) \le \phi(t)\, k_2(s), \quad \mbox{for a.e. } (t,s)  \in I \times I.$$
\end{enumerate}

Then $G[a+\lambda,T]$ is nonnegative on $I \times I$ if and only if $\lambda\in (\lambda_1(a,T),\infty)$ or  $\lambda\in (\lambda_1(a,T),\\ \bar{\mu}(a,T)]$, with $\lambda_1(a,T) <0$ the first eigenvalue of operator $L[a]$ in $X$ and $\bar{\mu}(a,T) \ge 0$ such that $L[a +\bar{\mu}]$ is nonresonant in $X$ and the related nonnegative Green's function $G[a+\bar{\mu},T]$ vanishes at some point of the square $I\times I$.
\end{lemma}

It is obvious that if the Green's function is strictly positive (resp. strictly negative) on $I \times I$ then condition $(P_g)$ (resp. $(N_g)$) is trivially fulfilled.
	
\section{Main results: how to decompose Green's functions as combination of periodic ones}	

We will study now different separated boundary conditions, and look for the connection between them and the periodic problem. In \cite{bacaig} some comparison principles are developed for these kind of boundary conditions. There it is proved that the validity of MP or AMP for one boundary condition is deduced from the validity of another one, considering for that some more restrictive hypothesis over the coefficients of the equation. 

Before studying each problem separately, we prove a necessary condition that must be satisfied by the Green's function of a self-adjoint operator. This result generalizes the one obtained for the periodic case in \cite[Lema 2.2]{cacid} and it is valid for periodic, Neumann and Dirichlet problems.

\begin{proposition}\label{p-Gpos}
Assume that operator $L[a]$ is nonresonant and self-adjoint on $X$. If the Green's function $G[a,T]$ does not change sign on $I\times I$ and $G[a,T]$ vanishes at some point $(t_0,s_0)\in I\times I$, then either $(t_0,s_0)$ belongs to the diagonal of the square $I\times I$ or $(t_0,s_0)$ is in the boundary of $I\times I$, that is, at least one of the three following properties hold:
\begin{enumerate}
\item $t_0=s_0\in I$.
\item $t_0=0$ or $t_0=T$.
\item $s_0=0$ or $s_0=T$.
\end{enumerate}
\end{proposition}

\begin{proof}
Suppose, on the contrary, that $G[a,T](t_0,s_0)=0$ for some $(t_0,s_0)\in (0,T)\times (0,T)$ such that $t_0\neq s_0$. Since $G[a,T](t_0,s_0)=G[a,T](s_0,t_0)$, we may assume $t_0>s_0$. 

By definition of the Green's function, we know that
$$x(t)\equiv G[a,T](t,s_0), \qquad t \in I,$$
solves the equation
$$x^{\prime\prime}(t)+a(t)\,x(t)=0,\quad \mbox{ a. e.  }\;t\in (s_0,T],\quad x(t_0)=x^\prime(t_0)=0.$$

Then, $G[a,T](t,s_0)=0$ for all $t\in(s_0,T]$ and, in consequence, from the symmetric property, $G[a,T](s_0,s)=0$ for all $s\in (s_0,T]$.

Now, fix $s\in (s_0,T]$, since $G[a,T]$ is nonnegative on $I \times I$, we have that function
$$y(t)\equiv G[a,T](t,s), \quad t \in I,$$
is a solution of
$$y^{\prime\prime}(t)+a(t)\,y(t)=0,\quad \mbox{ a. e.  }\;t\in [0,s),\quad y(s_0)=y^\prime(s_0)=0.$$

Once again, $G[a,T](t,s)=0$ for all $s\in(s_0,T]$ and all $t\in[0,s)$.

From symmetry, we deduce $G[a,T](t,s)=0 \mbox{ for all } t\in(s_0,T] \mbox{ and } s\in[0,t).$
This contradicts the fact that
$$\frac{\partial\,G[a,T]}{\partial\,t}\,(t^+,t)-\frac{\partial\,G[a,T]}{\partial\,t}\,(t^-,t)=1$$
for all $t\in(s_0,T)$ and so we deduce the result.
\end{proof}

If we consider the periodic case with $a(t)=\left(\frac{\pi}{T}\right)^2$, using \cite{CCiM} we obtain the following expression for the Green's function
$$G[a,T](t,s)=\frac{T}{2\,\pi}\left\{\begin{array}{ll}
\sin\left(\frac{\pi\,(t-s)}{T}\right), & \quad 0\le s\le t\le T, \\
\vspace*{-0.5cm}\\
\sin\left(\frac{\pi\,(t-s+T)}{T}\right), & \quad 0\le t< s\le T, 
\end{array} \right.$$
so the Green's function is strictly positive on $I \times I$ except for the diagonal and the points $(0,T)$ and $(T,0)$.

On the other hand, when $a(t)=k^2<\left(\frac{\pi}{T}\right)^2$ and the Dirichlet boundary conditions are studied, we have that the Green's function is
$$G[a,T](t,s)=\frac{1}{k\,\sin{(k\,T)}}\left\{\begin{array}{ll}
\sin{(k\,s)}\,\sin{(k\,(t-T))}, & \quad 0\le s\le t\le T, \\
\vspace*{-0.5cm}\\
\sin{(k\,t)}\,\sin{(k\,(s-T))}, & \quad 0\le t< s\le T. 
\end{array} \right.$$
We observe that $G[a,T]$ is strictly negative on $(0,T)\times (0,T)$ and vanishes on the boundary of its square of definition.

In consequence, the previous result cannot be improved for general self-adjoint Hill's operators.

\subsection{Neumann Problem}
In this section we will obtain the expression of the Green's function of Neumann problem both as sum of Green's functions of periodic problems and as sum of Green's functions of Neumann problem defined in a different interval.

Assume now that the Neumann boundary value problem
\begin{equation}\label{e-N}\tag{$N,\,T$}
L[a]\,u(t)=\sigma(t),\quad \mbox{ a. e.  }\;t\in I,\quad u^\prime(0)=u^\prime(T)=0
\end{equation}
has a unique solution $u \in W^{2,1}(I)$ for all $\sigma\in L^1(I)$.

Suppose, in addition, that the periodic boundary value problem
\begin{equation}
\label{e-P-2T}\tag{$P,\,2\,T$} L[\tilde a]\,u(t)=\tilde \sigma(t),\quad \mbox{ a. e.   }\; t\in J,\quad
u(0)=u(2\,T),\ u'(0)=u'(2\,T),
\end{equation}
with $\tilde{a}$ the even extension of $a$ to the interval $[0,2\,T]$, has a unique solution $v \in W^{2,1}(J)$ for all $\tilde \sigma \in L^1(J)$.

Let $u$ be the unique solution of problem \eqref{e-N}. It is clear that, by defining $v$ as the even extension of $u$, we have that $v\in W^{2,1}(J)$ is a solution of \eqref{e-P-2T} for the particular case of  $\tilde{\sigma}$ defined as the even extension of $\sigma$.

So, denoting by $G_N[a,T]$ the Green's function related to problem (\ref{e-N}), we have that for all $t\in I$ the following property holds:
\begin{eqnarray*}
\int_0^TG_N[a,T](t,s)\,\sigma(s)\,ds&=&u(t)=v(t)=\int_0^{2\,T}G_P[\tilde a,2\,T](t,s)\,\tilde{\sigma}(s)\,ds \\
&& \vspace*{-0.5cm}\\
&=&\int_0^T G_P[\tilde a,2\,T](t,s)\,\sigma(s)\,ds+ \int_T^{2\,T} G_P[\tilde a,2\,T](t,s)\,\sigma(2\,T-s)\,ds\\
&& \vspace*{-0.5cm}\\
&=&\int_0^T{\left(G_P[\tilde a,2\,T](t,s)+G_P[\tilde a,2\,T](t,2\,T-s)\right)\,\sigma(s)\,ds}.
\end{eqnarray*}

Since $\sigma\in L^1(I)$ is arbitrarily chosen and due to the uniqueness of the Green's function of both problems, we arrive at the following connecting expression between Neumann and periodic Green's functions
\begin{equation}
\label{e-Green-NP}
G_N[a,T](t,s)=G_P[\tilde a,2\,T](t,s)+G_P[\tilde a,2\,T](t,2\,T-s) \qquad \forall \;(t,s) \in I \times I.
\end{equation}

Note that the previous equality gives us the exact expression of the Green's function for the Neumann problem by means of the value of the periodic one.

On the other hand, let $\tilde \sigma\in L^1(J)$ be arbitrarily chosen. Since the periodic problem (\ref{e-P-2T}) has a unique solution $v$, we have, due to the fact that $\tilde a(t)=\tilde a(2\,T-t)$, that $w(t)=v(2\,T-t)$ is the unique solution of the periodic problem
$$ L[\tilde{a}]\,w(t)=\tilde \sigma(2\,T-t), \quad \mbox{a. e. }\;t \in J,\quad w(0)=w(2\,T),\; w^\prime(0)=w^\prime(2\,T).$$

Since
$$
v(2\,T-t)=\int_0^{2\,T}G_P[\tilde a,2\,T](2\,T-t,s)\,\tilde \sigma(s)\,ds
$$
and
$$w(t)=\int_0^{2\,T}G_P[\tilde a,2\,T](t,s)\,\tilde \sigma(2\,T-s)\,ds=\int_0^{2\,T}G_P[\tilde a,2\,T](t,2\,T-s)\,\tilde \sigma(s)\,ds,
$$
we arrive at
$$G_P[\tilde a,2\,T](2\,T-t,s)=G_P[\tilde a,2\,T](t,2\,T-s)\qquad \forall \;(t,s)\in J\times J$$
or, which is the same,
\begin{equation*}
G_P[\tilde a,2\,T](t,s)=G_P[\tilde a,2\,T](2\,T-t,2\,T-s)\qquad \forall \;(t,s)\in J \times J.
\end{equation*}

In particular, we have that equation (\ref{e-Green-NP}) can be rewritten as
\begin{equation}\label{e-Green-NP-2}
G_N[a,T](t,s)=G_P[\tilde a,2\,T](t,s)+G_P[\tilde a,2\,T](2\,T-t,s) \qquad \forall \; (t,s)\in I \times I.
\end{equation}

Moreover, we observe that $v$ also satisfies Neumann boundary conditions on $[0,2\,T]$ so it is the unique solution of the problem
\begin{equation}\label{e-N-2T}\tag{$N,\,2\,T$}
u^{\prime\prime}(t)+\tilde{a}(t)\,u(t)=\tilde{\sigma}(t),\quad \mbox{ a. e.   }\;t\in J,\quad u^{\prime}(0)=u^{\prime}(2\,T)=0.
\end{equation}
An analogous reasoning lets us conclude that
\begin{equation}\label{e-Green-N-N}
G_N[a,T](t,s)=G_N[\tilde{a},2\,T](t,s)+G_N[\tilde{a},2\,T](2\,T-t,s) \qquad \forall \; (t,s)\in I \times I
\end{equation}
or, using (\ref{e-Green-NP-2}),
\begin{eqnarray*}
\nonumber
G_N[a,T](t,s)&=& G_P[\tilde{\tilde{a}},4\,T](t,s)+G_P[\tilde{\tilde{a}},4\,T](4\,T-t,s) +G_P[\tilde{\tilde{a}},4\,T](2\,T-t,s)\\
\label{e-Green-N-P-4T}
\\
\nonumber
&&+G_P[\tilde{\tilde{a}},4\,T](2\,T+t,s) \qquad \forall \; (t,s)\in I\times I,
\end{eqnarray*}
with $\tilde{\tilde{a}}$ the even extension to the interval $[0,4\,T]$ of potential $\tilde{a}$. Note that $\tilde{\tilde{a}}$ is a $2\,T$-periodic function.

Observe that both $a$ and $\tilde{a}$ or $\tilde{\tilde{a}}$ are not necessarily continuous functions.

\subsection{Dirichlet Problem}
In the same way, we will obtain the Green's function for the Dirichlet problem both as a sum of Green's functions for the periodic one and as a sum of Green's functions for another Dirichlet problem.

Let $G_D[a,T](t,s)$ be the Green's function related to the Dirichlet boundary value problem
\begin{equation}\label{e-D}\tag{$D,\,T$}
u^{\prime\prime}+a(t)\,u(t)=\sigma(t),\quad \mbox{ a. e.  }\;t\in I,\quad u(0)=u(T)=0.
\end{equation}

Making an odd extension $v$ of function $u$ to the interval $J$, we deduce that $v$ is a solution of \eqref{e-P-2T} for the particular choice of $\tilde{\sigma}$ as the odd extension of function $\sigma$ to the interval $J$.

Reasoning as in the previous case, we conclude that
\begin{equation}\label{e-Green-DP}
G_D[a,T](t,s)=G_P[\tilde a,2\,T](t,s)-G_P[\tilde a,2\,T](2\,T-t,s) \qquad \forall \; (t,s)\in I\times I.
\end{equation}

On the other hand, $v$ also satisfies Dirichlet boundary conditions on $[0,2\,T]$ so it is a solution of the problem
\begin{equation}\label{e-D-2T}\tag{$D,\,2\,T$}
u^{\prime\prime}(t)+\tilde{a}(t)\,u(t)=\tilde{\sigma}(t),\quad \mbox{ a. e.   }\;t\in J,\quad u(0)=u(2\,T)=0
\end{equation}
and we arrive at
\begin{equation}\label{e-Green-D-D}
G_D[a,T](t,s)=G_D[\tilde{a},2\,T](t,s)-G_D[\tilde{a},2\,T](2\,T-t,s) \qquad \forall \; (t,s)\in I \times I
\end{equation}
or, using (\ref{e-Green-DP}),
\begin{eqnarray*}
\nonumber
G_D[a,T](t,s)&=& G_P[\tilde{\tilde{a}},4\,T](t,s)-G_P[\tilde{\tilde{a}},4\,T](4\,T-t,s) -G_P[\tilde{\tilde{a}},4\,T](2\,T-t,s)\\
\label{e-Green-D-P-4T}
\\
\nonumber
&&+G_P[\tilde{\tilde{a}},4\,T](2\,T+t,s) \qquad \forall \; (t,s)\in I\times I.
\end{eqnarray*}

\subsection{Relation between Neumann and Dirichlet problems}
From the expressions previously obtained for $G_N[a,T]$ and $G_D[a,T]$, we are going to connect the existence and uniqueness of solution, and consequently the spectrum, for Neumann, Dirichlet and periodic problems. 

As an immediate consequence of (\ref{e-Green-NP-2}) and (\ref{e-Green-DP}) we have
\begin{equation}
\label{e-Green-DN-P}
G_N[a,T](t,s)+G_D[a,T](t,s)=2\,G_P[\tilde a,2\,T](t,s) \qquad \forall \; (t,s)\in I\times I
\end{equation}
and
\begin{equation}\label{e-N-D-P}
G_N[a,T](t,s)-G_D[a,T](t,s)=2\,G_P[\tilde a,2\,T](2\,T-t,s) \qquad \forall \; (t,s)\in I\times I.
\end{equation}

\begin{remark}
\label{r-per->Neu}
From equality \eqref{e-Green-NP} we have that if problem \eqref{e-P-2T} has a unique solution, then problem \eqref{e-N} has a solution given by
$$u(t)=\int_0^T{\left(G_P[\tilde a,2\,T](t,s)+G_P[\tilde a,2\,T](2\,T-t,s)\right)\,\sigma(s)\, ds}.$$

The uniqueness follows from the fact that the Neumann boundary conditions are linearly independent (see \cite[Lemma 1.2.21]{Cab}).

Consequently we observe that if problem \eqref{e-P-2T} is nonresonant, the same holds for \eqref{e-N}. In other words, the sequence of eigenvalues of problem \eqref{e-N} is contained into the sequence of eigenvalues of \eqref{e-P-2T}.

The same argument is valid, by means of equality \eqref{e-Green-DP}, to ensure that if problem \eqref{e-P-2T} has a unique solution, then problem \eqref{e-D} has a unique solution too. 

As consequence, denoting by $\Lambda_N[a,T]$, $\Lambda_D[a,T]$ and $\Lambda_P[\tilde{a},2\,T]$ the corresponding set of eigenvalues of problems \eqref{e-N}, \eqref{e-D} and \eqref{e-P-2T}, we deduce that
$$\Lambda_N[a,T] \cup \Lambda_D[a,T] \subset \Lambda_P[\tilde{a},2\,T].$$

On the other hand, equations \eqref{e-Green-DN-P} and \eqref{e-N-D-P} imply that the uniqueness of solution of problems \eqref{e-N} and \eqref{e-D} warrants the uniqueness of solution of \eqref{e-P-2T}. 

Thus, we conclude that
$$\Lambda_N[a,T] \cup \Lambda_D[a,T] = \Lambda_P[\tilde{a},2\,T].$$

The same reasoning lets us also deduce the following facts:
\begin{itemize}
\item If problem (\ref{e-N-2T}) has a unique solution then (\ref{e-N}) has a unique solution too.
\item If problem (\ref{e-D-2T}) has a unique solution then (\ref{e-D}) has a unique solution too.
\item The existence and uniqueness of solution of problem
\begin{equation}\label{e-P-4T}\tag{$P,\,4\,T$}
u^{\prime\prime}(t)+\tilde{\tilde{a}}(t)\,u(t)=\sigma(t),\ \mbox{ a. e. }\;t\in [0,4\,T],\quad u(0)=u(4\,T),\ u^{\prime}(0)=u^{\prime}(4\,T)
\end{equation}
is equivalent to the existence and uniqueness of solution of (\ref{e-N-2T}) and (\ref{e-D-2T}) and, consequently, implies the existence and uniqueness of solution of (\ref{e-N}) and (\ref{e-D}).
\end{itemize}

As consequence, denoting by $\Delta_N[\tilde{a},2\,T]$,\, $\Delta_D[\tilde{a},2\,T]$ and $\Delta_P[\tilde{\tilde{a}},4\,T]$ the set of eigenvalues of (\ref{e-N-2T}), (\ref{e-D-2T}) and (\ref{e-P-4T}), respectively, we have that
$$\Delta_N[a,T] \subset \Delta_N[\tilde{a},2\,T],$$
$$\Delta_D[a,T] \subset \Delta_D[\tilde{a},2\,T]$$
and
$$\Delta_N[a,T] \cup \Delta_D[a,T] \subset \Delta_N[\tilde{a},2\,T] \cup \Delta_D[\tilde{a},2\,T] = \Delta_P[\tilde{\tilde{a}},4\,T].$$
\end{remark}

\subsection{Mixed Problems and their relation with Neumann and Dirichlet}
\label{s-mixed}
Consider now the Mixed boundary value problem
\begin{equation}\label{e-M-1}\tag{$M_1,\,T$}
u^{\prime\prime}(t)+a(t)\,u(t)=\sigma(t),\ \mbox{ a. e.  }\;t\in I,\quad u'(0)=u(T)=0,
\end{equation}
and denote by $G_{M_1}[a,T](t,s)$ its related Green's function. 

Making an odd extension $v$ of function $u$ to the interval $J$, we deduce that $v$ is a solution of the anti-periodic problem
\begin{equation}\label{e-A-2T}\tag{$A,\,2\,T$}
u^{\prime\prime}(t)+\tilde{a}(t)\,u(t)=\tilde{\sigma}(t),\  \mbox{ a. e.  }\;t\in J,\quad u(0)=-u(2\,T),\quad u^{\prime}(0)=-u^{\prime}(2\,T),
\end{equation}
with $\tilde{\sigma}$ the odd extension of $\sigma$ to the interval $J$.

Thus we conclude that
\begin{equation}\label{e-Green-M1-A}
G_{M_1}[a,T](t,s)= G_A[\tilde{a},2\,T](t,s)- G_A[\tilde{a},2\,T](2\,T-t,s) \qquad \forall \; (t,s)\in I\times I.
\end{equation}

But $v$ also satisfies Neumann conditions on $[0,2\,T]$ so it is a solution of (\ref{e-N-2T}) and
\begin{equation}\label{e-Green-M1-N}
G_{M_1}[a,T](t,s)= G_N[\tilde{a},2\,T](t,s)- G_N[\tilde{a},2\,T](2\,T-t,s) \qquad \forall \; (t,s)\in I\times I,
\end{equation}
or, equivalently, using equation (\ref{e-Green-NP-2}) 
\begin{eqnarray*}
\nonumber
G_{M_1}[a,T](t,s)&=& G_P[\tilde{\tilde{a}},4\,T](t,s)+G_P[\tilde{\tilde{a}},4\,T](4\,T-t,s)-G_P[\tilde{\tilde{a}},4\,T](2\,T-t,s)\\
\label{e-Green-M1-P}
\vspace*{-1cm}\\
\nonumber
&&-G_P[\tilde{\tilde{a}},4\,T](2\,T+t,s) \qquad \forall \; (t,s)\in I\times I.
\end{eqnarray*}

\vspace*{0.2cm}

We can study now the corresponding Mixed boundary value problem
\begin{equation}\label{e-M-2}\tag{$M_2,\,T$}
u^{\prime\prime}(t)+a(t)\,u(t)=\sigma(t),\quad \mbox{ a. e.  }\;t\in I,\quad u(0)=u'(T)=0,
\end{equation}
denoting by $G_{M_2}[a,T](t,s)$ its related Green's function.

Considering in this case $v$ as the even extension of function $u$ to the interval $J$, we conclude that $v$ is a solution of problem (\ref{e-A-2T}) with $\tilde{\sigma}$ the even extension of $\sigma$. Reasoning as in previous cases, we deduce that
\begin{equation}\label{e-Green-M2-A}
G_{M_2}[a,T](t,s)=G_A[\tilde{a},2\,T](t,s)+G_A[\tilde{a},2\,T](2\,T-t,s) \qquad \forall \; (t,s)\in I\times I.
\end{equation}

In this case, $v$ satisfies Dirichlet conditions on the interval $[0,2\,T]$ so it is also a solution of (\ref{e-D-2T}) and we have that
\begin{equation}\label{e-Green-M2-D}
G_{M_2}[a,T](t,s)=G_D[\tilde{a},2\,T](t,s)+G_D[\tilde{a},2\,T](2\,T-t,s) \qquad \forall \; (t,s)\in I\times I,
\end{equation}
or, which is the same, using equation (\ref{e-Green-DP})
\begin{eqnarray*}
\nonumber
G_{M_2}[a,T](t,s)&=&G_P[\tilde{\tilde{a}},4\,T](t,s)-G_P[\tilde{\tilde{a}},4\,T](4\,T-t,s) +G_P[\tilde{\tilde{a}},4\,T](2\,T-t,s)\\
\label{e-Green-M2-P}\\
\nonumber
&&-G_P[\tilde{\tilde{a}},4\,T](2\,T+t,s) \qquad \forall \; (t,s)\in I\times I.
\end{eqnarray*}

Consequently, from (\ref{e-Green-M1-A}) and (\ref{e-Green-M2-A}) we deduce the following properties
$$G_{M_2}[a,T](t,s)+G_{M_1}[a,T](t,s)=2 \, G_A[\tilde{a},2\,T](t,s)\qquad \forall \; (t,s)\in I\times I$$
and
$$G_{M_2}[a,T](t,s)-G_{M_1}[a,T](t,s)=2\,G_A[\tilde{a},2\,T](2\,T-t,s)\qquad \forall \; (t,s)\in I\times I.$$

Analogously, from (\ref{e-Green-N-N}) and (\ref{e-Green-M1-N}) we obtain that
\begin{equation}\label{e-Green-NM1-N}
G_{N}[a,T](t,s)+G_{M_1}[a,T](t,s)=2\, G_N[\tilde{a},2\,T](t,s)\qquad \forall \; (t,s)\in I\times I
\end{equation}
and
\begin{equation}\label{e-N-M1-N}
G_{N}[a,T](t,s)-G_{M_1}[a,T](t,s)=2\, G_N[\tilde{a},2\,T](2\,T-t,s)\qquad \forall \; (t,s)\in I\times I
\end{equation}
and, from (\ref{e-Green-D-D}) and (\ref{e-Green-M2-D}) we deduce
\begin{equation}\label{e-Green-M2D-D}
G_{M_2}[a,T](t,s)+G_{D}[a,T](t,s)=2\, G_D[\tilde{a},2\,T](t,s)\qquad \forall \; (t,s)\in I\times I
\end{equation}
and
\begin{equation}\label{e-M2-D-D}
G_{M_2}[a,T](t,s)-G_{D}[a,T](t,s)=2\, G_D[\tilde{a},2\,T](2\,T-t,s)\qquad \forall \; (t,s)\in I\times I.
\end{equation}

Compiling all the previous equations we arrive at
$$G_{N}[a,T](t,s)+G_D[a,T](t,s)+G_{M_1}[a,T](t,s)+G_{M_2}[a,T](t,s)=4\,G_P[\tilde{\tilde{a}},4\,T](t,s)\quad \forall \, (t,s)\in I\times I.$$

\begin{remark} \label{r-Ant->Mixed}
Reasoning analogously to Remark \ref{r-per->Neu} and taking into account the fact that all the considered boundary conditions are linearly independent, using \cite[Lemma 1.2.21]{Cab} we are able to deduce the following facts:
\begin{itemize}
\item The existence and uniqueness of solution of problems (\ref{e-M-1}) and (\ref{e-M-2}) is equivalent to the existence and uniqueness of solution of (\ref{e-A-2T}).
\item If problem (\ref{e-P-4T}) has a unique solution, then (\ref{e-M-1}) and (\ref{e-M-2}) also have a unique solution and, consequently, problem (\ref{e-A-2T}) has a unique solution too.
\item The existence and uniqueness of solution of problems (\ref{e-N}) and (\ref{e-M-1}) is equivalent to the existence and uniqueness of solution of (\ref{e-N-2T}).
\item The existence and uniqueness of solution of problems (\ref{e-D}) and (\ref{e-M-2}) is equivalent to the existence and uniqueness of solution of (\ref{e-D-2T}).
\item From Remark \ref{r-per->Neu}, we know that the existence and uniqueness of solution of (\ref{e-N-2T}) and (\ref{e-D-2T}) is equivalent to the existence and uniqueness of solution of problem (\ref{e-P-4T}). We deduce then that the periodic problem (\ref{e-P-4T}) has a unique solution if and only if problems (\ref{e-N}), (\ref{e-D}), (\ref{e-M-1}) and (\ref{e-M-2}) have a unique solution.
\end{itemize}

As consequence, denoting by $\Lambda_{M_1}[a,T]$, $\Lambda_{M_2}[a,T]$ and $\Lambda_A[\tilde{a},2\,T]$ the set of eigenvalues of problems (\ref{e-M-1}), (\ref{e-M-2}) and (\ref{e-A-2T}), respectively, we conclude that
$$\Lambda_{M_1}[a,T]\cup \Lambda_{M_2}[a,T]= \Lambda_A[\tilde{a},2\,T] \subset \Lambda_P[\tilde{\tilde{a}},4\,T],$$
$$\Lambda_N[a,T] \cup \Lambda_{M_1}[a,T]= \Lambda_N[\tilde{a},2\,T],$$
$$\Lambda_D[a,T] \cup \Lambda_{M_2}[a,T]= \Lambda_D[\tilde{a},2\,T]$$
and
$$\Lambda_N[a,T]\cup \Lambda_D[a,T]\cup \Lambda_{M_1}[a,T]\cup \Lambda_{M_2}[a,T]= \Lambda_P[\tilde{\tilde{a}},4\,T].$$

Moreover, taking into account the previous inclusions and the ones obtained in Remark \ref{r-per->Neu} it is clear that
$$\Lambda_P[\tilde{a},2\,T]=\Lambda_N[a,T]\cup \Lambda_D[a,T] \subset \Lambda_P[\tilde{\tilde{a}},4\,T]$$
and, consequently, we deduce
$$\Lambda_P[\tilde{a},2\,T] \cup \Lambda_A[\tilde{a},2\,T]= \Lambda_P[\tilde{\tilde{a}},4\,T].$$
\end{remark}

\vspace*{0.3cm}

On the other hand, it is possible to obtain a direct relation between the Green's functions of the two mixed problems

\begin{lemma}
Let $a\in L^1(I)$ and define $b(t)=a(T-t)$ for all $t\in I$. Then
$$G_{M_1}[a,T](T-t,T-s)=G_{M_2}[b,T](t,s) \qquad \forall \; (t,s)\in I\times I.$$
\end{lemma}

\begin{proof}
Let $u$ be the unique solution for the mixed problem \eqref{e-M-1}, given explicitly by 
$$u(t)=\int_{0}^{T}G_{M_1}[a,T](t,s)\,\sigma(s)\,ds.$$
Clearly, $v(t)=u(T-t)$ is the unique solution of the mixed problem
$$v''(t)+b(t)\,v(t)=\sigma(T-t), \quad v(0)=v'(T)=0.$$
Therefore, we know that
$$u(T-t)=\int_{0}^{T}G_{M_1}[a,T](T-t,s)\,\sigma(s)\,ds,$$
meanwhile, on the other hand, we have that
$$v(t)=\int_{0}^{T}G_{M_2}[b,T](t,s)\,\sigma(T-s)\,ds=\int_{0}^{T}G_{M_2}[b,T](t,T-s)\,\sigma(s)\,ds.$$
As the previous equalities are valid for all $\sigma \in L^1(I)$ we deduce that
$$G_{M_1}[a,T](T-t,s)=G_{M_2}[b,T](t,T-s) \qquad \forall \; (t,s)\in I\times I$$
or, equivalently,
$$G_{M_1}[a,T](T-t,T-s)=G_{M_2}[b,T](t,s) \qquad \forall \; (t,s)\in I\times I.$$
\end{proof}

As an immediate consequence we have that

\begin{corollary}
\label{l-a(T-t)}
Let $a \in L^1(I)$ and define $b(t)=a(T-t)$ for all $t \in I$. Then
$$\Lambda_{M_1}[a,T]=\Lambda_{M_2}[b,T].$$
In particular, if $a(t)=a(T-t)$ then the eigenvalues of mixed problems are the same.
\end{corollary}

\section{Some applications}

\subsection{Constant sign Green's function} \label{sec-signo-cte}

We start with the adaptation of Proposition \ref{p-Gpos} for the Green's function of the Neumann problem. Next we prove the existence of a certain order of appearance for the first eigenvalues of each boundary value problem and we relate the constant sign of the diverse Green's functions considered in previous sections. As a consequence we will deduce some comparison results between the Green's function related to different boundary value problems. Moreover some order between the different spectrums is obtained. We finalize with several examples that illustrate the obtained results.

Let ${\lambda}_N(a,T)$ be the smallest eigenvalue of the
Neumann problem
$$
u''(t)+(a(t) + \lambda)\, u(t)=0,\quad \mbox{ a. e.  }\;t\in I, \quad u'(0)=u'(T)=0.
$$

Let ${\lambda}_{M_1}(a,T)$ be the smallest eigenvalue of the Mixed problem
$$
u''(t)+(a(t) + \lambda)\, u(t)=0,\quad \mbox{ a. e.  }\;t\in I, \quad u'(0)=u(T)=0.
$$

Let ${\lambda}_{M_2}(a,T)$  be the smallest eigenvalue of the Mixed problem
$$
u''(t)+(a(t) + \lambda)\, u(t)=0,\quad \mbox{ a. e.  }\;t\in I, \quad u(0)=u'(T)=0.
$$

\begin{lemma}\label{l-Green-NP}
Suppose that the Green's function $G_N[a,T]$ is nonnegative on $I \times I$ and there is some $(t_0,s_0) \in I \times I$ for which $G_N[a,T](t_0,s_0)=0$, then either $t_0=s_0=0$ or $t_0=s_0=T$.
\end{lemma}
\begin{proof}
Suppose that $G_N[a,T](t_0,s_0)=0$ for some $(t_0,s_0)\in I\times I$. Since $G_N[a,T]\ge 0$ on $I\times I$, as operator $L[a]$ is self-adjoint, Proposition \ref{p-Gpos} lets us conclude that either $(t_0,s_0)$ belongs to the boundary of the square of definition or to its diagonal.

In the first case, suppose that $t_0 \in (0, T)$ and $s_0=0$. Then we have that $x_0(t)\equiv G_N[a,T](t,0)$ satisfies the equation 
$$x''_0(t)+a(t)\,x_0(t)=0, \ t\in(0,T], \quad x_0(t_0)=x'_0(t_0)=0,$$ 
which means that $G_N[a,T](t,0)\equiv 0$ on $(0,T]$.

From the symmetry of $G_N[a,T]$, we have that $G_N[a,T](0,s)\equiv 0$ for all $s \in (0,T]$.

As consequence, $x_s(t)\equiv G_N[a,T](t,s)$ satisfies the equation
$$x_s''(t)+a(t)\, x_s(t)=0, \ t \in [0,s), \quad x_s(0)=x_s'(0)=0,$$
which implies that $G_N[a,T](t,s)\equiv 0$ for all $t<s$. 
Using again the symmetry of $G_N[a,T]$ we have that it is identically zero on $I \times I$ and we reach a contradiction.

The argument is valid for all $(t_0,s_0)$ in the boundary of $I \times I$ except for $(0,0)$ and $(T,T)$.

Assume now that $G_N[a,T](t_0,t_0)=0$ for some $t_0\in (0,T)$. In this case, defining $x_{t_0}(t)$ as the even extension to $J$ of $G_N[a,T](t,t_0)$, we have that it satisfies the equation
$$x_{t_0}''(t)+\tilde a(t)\, x_{t_0}(t)=0, \; t \in (t_0,2\, T-t_0), \quad x_{t_0}(t_0)=x_{t_0}(2\, T -t_0)=0.$$

From classical Sturm-Liouville theory (\cite{simmons}), we have that for any $\lambda\ge 0$ every nontrivial solution of the equation
\begin{equation}\label{e-a+lambda}
y''(t)+(\tilde a(t)+\lambda)\, y(t)=0, \ t \in \mathbb{R},
\end{equation}
has as least one zero on $[t_0,2\, T-t_0]$.

Then, as the even extension to $J$ of the positive eigenfunction on $(0,T]$ associated with $\lambda_{M_2}(a,T)$ solves (\ref{e-a+lambda}), we deduce that $\lambda_{M_2}(a,T)<0$. 

As consequence, for any $\lambda \in ({\lambda}_{M_2}(a,T),0]$ we have that $y_0$, the even extension to $J$ of $G_N[a+\lambda,T](t,0)$, has at least one zero on $(0,2\,T)$. Moreover, all the zeros of $y_0$ are simple because otherwise $G_N[a+\lambda,T](t,0)\equiv 0$ on $(0,T]$, which cannot happen. Then necessarily $y_0$ changes its sign on $(0,2\,T)$ and, as it is an even function, $G_N[a+\lambda,T](t,0)$ changes its sign on $(0,T)$. This contradicts the hypothesis that $G_N[a,T]$ is nonnegative on $I \times I$.
\end{proof}

Note that if $G_N[a,T](0,0)=0$ we have that $x(t)\equiv G_N[a,T](t,0)$ is a solution of
\begin{equation}\label{e-M2-h}
x^{\prime\prime}(t)+a(t)\,x(t)=0,\ t\in I,\quad x(0)=x^\prime(T)=0.
\end{equation}

Moreover, when $G_N[a,T](T,T)=0$, $y(t)=G_N[a,T](t,T)$ is a solution of
\begin{equation}\label{e-M1-h}
y^{\prime\prime}(t)+a(t)\,y(t)=0,\ t\in I,\quad y^\prime(0)=y(T)=0.
\end{equation}

As consequence, we deduce the following result
\begin{corollary}\label{c-Green-N}
If $G_P[\tilde a,2\,T]$ has constant sign on $J\times J$ then $G_N[a,T]$ has the same sign as $G_P[\tilde a,2\,T]$ on $I \times I$ and it is different from zero for all $(t,s)\in (I\times I)\setminus\{(0,0)\cup(T,T)\}$. 

Moreover, $G_N[a,T](T,T)=0$ if and only only if equation (\ref{e-M2-h}) has a non zero and constant sign solution on $[0,T)$, which means that $\lambda_{M_2}(a,T)=0$.

$G_N[a,T](0,0)=0$ if and only if equation (\ref{e-M1-h}) has a non zero and constant sign solution on $(0,T]$, which means that $\lambda_{M_1}(a,T)=0$.
\end{corollary}

In the sequel we prove some relations between the first eigenvalues of different problems.

\begin{theorem}
\label{t-lambdaP=lambdaN}
The following equalities are fulfilled for any $a \in L^1(I)$.
\begin{enumerate}
\item ${\lambda}_N(a,T)={\lambda}_N(\tilde a,2\,T)={\lambda}_P(\tilde a,2\,T).$
\item $G_N[a+\lambda,T](t,s)<0$ for all $(t,s) \in I \times I$ if and only if $\lambda < {\lambda}_P(\tilde a,2\,T)$.
\item $G_N[a+\lambda,T](t,s)> 0$ on $I \times I$ if and only if  ${\lambda}_P(\tilde a,2\,T)< \lambda < \min{\{{\lambda}_{M_1}(a,T), {\lambda}_{M_2}(a,T)\}}$.
\item $G_N[a+\lambda,T](t,s)> 0$ for all $(t,s) \in (I \times I)\backslash \{(0,0) \cup (T,T)\}$ if and only if ${\lambda}_P(\tilde a,2\,T)< \lambda \le \min{\{{\lambda}_{M_1}(a,T), {\lambda}_{M_2}(a,T)\}}$.
\item $G_N[a,T](0,0)=2\,G_P[\tilde a,2\,T](0,0)$ and ${\lambda}_{M_1}(a,T)$ is characterized as the first root of equation $$(G_N[a+\lambda,T](0,0)=)\; G_P[\tilde a+\lambda,2\,T](0,0)=0.$$
\item $G_N[a,T](T,T)=2\,G_P[\tilde a,2\,T](T,T)$ and ${\lambda}_{M_2}(a,T)$ is characterized as the first root of equation $$(G_N[a+\lambda,T](T,T)=)\; G_P[\tilde a+\lambda,2\,T](T,T)=0.$$
\item ${\lambda}_A(\tilde a,2\,T) =\min{\{{\lambda}_{M_1}(a,T), {\lambda}_{M_2}(a,T)\}}.$
\item ${\lambda}_{M_2}(a,T)= {\lambda}_D(\tilde a,2\,T) <{\lambda}_D(a,T) .$
\item ${\lambda}_N(a,T)<{\lambda}_{M_1}(a,T).$
\end{enumerate}
\end{theorem}

\begin{proof}
Suppose that the periodic problem \eqref{e-P-2T} is uniquely solvable. From Remark \ref{r-per->Neu} we know that the Neumann problem \eqref{e-N} is uniquely solvable too.

From Lemmas \ref{l-Green-neg} and \ref{l-zhang-eigen-2}, we know that $G_P[\tilde a+\lambda,2\,T]$ is strictly negative on $J \times J$ if and only if $\lambda < {\lambda}_P(\tilde a,2\,T)$ and it is nonnegative on $J \times J$ if and only if ${\lambda}_P(\tilde a,2\,T)< \lambda \le {\lambda}_A(\tilde a,2\,T)$.

As consequence, equation \eqref{e-Green-NP-2} implies that $G_N[a+\lambda,T]$ is strictly negative on $I \times I$ (and consequently condition ($N_g$) is verified) for all  $\lambda < {\lambda}_P(\tilde a,2\,T)$ and it is nonnegative on $I \times I$ (and satisfies ($P_g$)) for all ${\lambda}_P(\tilde a,2\,T)< \lambda \le {\lambda}_A(\tilde a,2\,T)$. 

Since, as it is pointed out in Lemma \ref{l-M-NT}, the set of parameters $\lambda$ for which the maximum principle holds is connected, and its supremum is the first eigenvalue of the considered operator, we conclude that 
$G_N[a+\lambda,T]$ is strictly negative on $I \times I$ if and only if  $\lambda < {\lambda}_P(\tilde a,2\,T)$ and ${\lambda}_N(a,T)={\lambda}_P(\tilde a,2\,T)$.

Let $u_N$ be an eigenfunction associated to ${\lambda}_N(a,T)$. 
From classical spectral theory we know that ${\lambda}_N(a,T)$ is simple and $u_N$ is strictly positive on $(0,T)$. In fact it is strictly positive on the closed interval because, on the contrary, we would have that $u_N(0)=u'_N(0)=0$ (or $u_N(T)=u'_N(T)=0$) and we would conclude that $u_N$ is identically zero on $I$.

Considering now the even extension of $u_N$ to the interval $J$, we have that it remains an eigenfunction associated to the same value ${\lambda}_N(a,T)$ for the potential $\tilde a$ defined on the interval $J$. Since it is strictly positive on $J$, we deduce that ${\lambda}_N(a,T)$ is the smallest eigenvalue of \eqref{e-N-2T}, that is, ${\lambda}_N(a,T)={\lambda}_N(\tilde a,2\,T)$.

So the two first identities are proved.

On the other hand, we know that $G_N[a+\lambda,T]$ is  nonnegative on $I \times I$ for all ${\lambda}_P(\tilde a,2\,T)< \lambda \le {\lambda}_A(\tilde a,2\,T)$. So Lemma \ref{l-M-PT} implies that the maximum $\lambda$ for which the Green's function is nonnegative on $I \times I$ is not an eigenvalue of the Neumann problem. Lemma \ref{l-Green-NP} and Corollary \ref{c-Green-N} ensure that such value is exactly $\min{\{{\lambda}_{M_1}(a,T), {\lambda}_{M_2}(a,T)\}}$. So assertions three and four are proved.

Now, using equality \eqref{e-Green-NP-2} and the fact that $G_P[\tilde{a}+\lambda,2\,T](0,0)=G_P[\tilde{a}+\lambda,2\,T](2\,T,0)$, we conclude that for all $\lambda \in \R$ the following equalities hold
\begin{equation}
\label{e-GP-0}
G_N[a+\lambda,T](0,0)=2\, G_P[\tilde a+\lambda,2\,T](0,0),
\end{equation}
and
\begin{equation}
\label{e-GP-T}
G_N[a+\lambda,T](T,T)=2\, G_P[\tilde a+\lambda,2\,T](T,T).
\end{equation}

From Lemma \ref{l-Green-dec}, we have that while both values on equations \eqref{e-GP-0} and \eqref{e-GP-T} are positive, they are strictly decreasing with respect to $\lambda$. Thus, Corollary \ref{c-Green-N} ensures that ${\lambda}_{M_1}(a,T)$ is the first zero of \eqref{e-GP-0} and ${\lambda}_{M_2}(a,T)$ the first zero of \eqref{e-GP-T}. Then, assertions five and six hold.

Assertion seven is an immediate consequence of $\Lambda_A[\tilde{a},2\,T]=\Lambda_{M_1}[a,T]\cup \Lambda_{M_2}[a,T]$.

To see assertion eight, it is enough to consider the even extension to $J$ of the eigenfunction associated to ${\lambda}_{M_2}(a,T)$. Obviously it satisfies Dirichlet boundary conditions and is positive on $(0,2\,T)$. Thus, ${\lambda}_{M_2}(a,T)={\lambda}_{D}(\tilde a,2\,T)$.

Taking into account that if we consider the odd extension to $J$ of the eigenfunction related to the Dirichlet problem with ${\lambda}_{D}(a,T)$ (which is positive on $(0,T)$), we obtain a changing sign eigenfunction of \eqref{e-D-2T}. Consequently,  ${\lambda}_{D}(a,T)$ is an eigenvalue of this problem too, but it is not the least one because the associated eigenfunction has not constant sign on $(0,2\,T)$. As consequence, ${\lambda}_{D}(\tilde a,2\,T) <{\lambda}_{D}(a,T)$.

The same reasoning is valid to prove assertion nine. Indeed, the odd extension to $J$ of the eigenfunction related to $\lambda_{M_1}(a,T)$ (which is positive on $(0,T)$) is an eigenfunction of \eqref{e-N-2T} and changes its sign on $J$. Consequently $\lambda_{M_1}(a,T)>\lambda_N(\tilde{a},2\,T)=\lambda_N(a,T)$.
\end{proof}

\begin{remark}
Corollary \ref{l-a(T-t)} assures that both $\lambda_A(\tilde{a},2\,T)=\lambda_{M_1}(a,T)$ and $\lambda_A(\tilde{a},2\,T)=\lambda_{M_2}(a,T)$ are possible in assertion four of the previous theorem. Indeed, if for some potential $a$, $\lambda_{M_1}(a,T)<\lambda_{M_2}(a,T)$, then necessarily $\lambda_{M_2}(b,T)<\lambda_{M_1}(b,T)$ for $b(t)=a(T-t)$.
\end{remark}

For an arbitrary potential $a$ we obtain the following corollaries.

\begin{corollary}
\label{c-lambdaP=lambdaN}
The following equalities are fulfilled for any $a \in L^1(I)$.
\begin{enumerate}
\item $G_N[a,T](t,s)<0$ for all $(t,s) \in I \times I$ if and only if $0 < {\lambda}_P(\tilde a,2\,T) \; (={\lambda}_N(a,T))$.
\item $G_N[a,T](t,s)> 0$ for all $(t,s) \in (0,T) \times (0,T)$ if and only if $({\lambda}_N(a,T)=)\,{\lambda}_P(\tilde a,2\,T)< 0 \le \min{\{{\lambda}_{M_1}(a,T), {\lambda}_{M_2}(a,T)\}}\, (\le {\lambda}_D(\tilde a,2\,T) <{\lambda}_D(a,T) )$.
\end{enumerate}
\end{corollary}

\begin{corollary} \label{c-Green-M2-neg}
The following property holds for any $a\in L^1(I)$: 

$G_{M_2}[a+\lambda](t,s)<0$ for all $(t,s)\in (0,T] \times (0,T]$ if and only if $\lambda<\lambda_{M_2}(a,T)$. 
\end{corollary}

\begin{proof}
From Lemma \ref{l-coppel} we know that $G_D[\tilde{a}+\lambda, 2\,T]$ is strictly negative on $(0, 2\,T) \times (0,2\,T)$ if and only if $\lambda<\lambda_D(\tilde{a},2\,T)$.

Considering then  equation (\ref{e-Green-M2-D}), it is immediately deduced that if $\lambda<\lambda_D(\tilde{a},2\,T)$, $G_{M_2}[a,T]$ is strictly negative on $(0,T]\times (0,T]$. We conclude the result from Lemma \ref{l-M-NT} and the fact that $\lambda_D(\tilde{a},2\,T)=\lambda_{M_2}(a,T)$.
\end{proof}

Moreover, as an immediate consequence of Corollaries \ref{l-a(T-t)} and \ref{c-Green-M2-neg} we have

\begin{corollary}\label{c-Green-M1-neg}
Let $a\in L^1(I)$, then:

$G_{M_1}[a+\lambda](t,s)<0$ for all $(t,s)\in [0,T) \times [0,T)$ if and only if $\lambda<\lambda_{M_1}(a,T)$.
\end{corollary}

From the previous results we deduce some relations between the constant sign of the Green's functions.

\begin{corollary}
\label{c-GP<0->GD<0}
For any $a \in L^1(I)$ we have the following properties.
\begin{enumerate}
\item $G_P[\tilde a,2\,T]<0$ on $J \times J$ if and only if $G_N[a,T]<0$ on $I \times I$. This is equivalent to $G_N[\tilde{a},2\,T]<0$ on $J\times J$.
\item $G_P[\tilde a,2\,T]> 0$ on $(0,2\,T) \times (0,2\,T)$ if and only if $G_N[a,T]> 0$ on $(0,T) \times (0,T)$.
\item If $G_N[\tilde{a},2\,T]> 0$ on $(0,2\,T) \times (0,2\,T)$ then $G_N[a,T]>0$ on $(0,T) \times (0,T)$.
\item If $G_P[\tilde a,2\,T]<0$ on $J \times J$ then $G_D[\tilde a,2\,T]<0$ on $(0,2\,T) \times (0,2\,T)$.
\item If $G_P[\tilde{a},2\,T]> 0$ on $(0,2\,T) \times (0,2\,T)$ then $G_D[\tilde a,2\,T]<0$ on $(0,2\,T) \times (0,2\,T)$.
\item If $G_N[a,T]$ (or $G_P[\tilde{a},2\,T]$) has constant sign on $I\times I$, then $G_D[a,T]<0$ on $(0,T) \times (0,T)$, $G_{M_1}[a,T]<0$ on $[0,T)\times[0,T)$ and $G_{M_2}[a,T]<0$ on $(0,T] \times (0,T]$.
\item $G_D[\tilde{a},2\,T]<0$ on $(0,2\,T) \times (0,2\,T)$ if and only if $G_{M_2}[a,T]< 0$ on $(0,T]\times (0,T]$.
\item If $G_{M_2}[a,T]<0$ on $(0,T]\times (0,T]$ or $G_{M_1}[a,T]<0$ on $[0,T) \times [0,T)$ then $G_D[a,T]<0$ on $(0,T) \times (0,T)$.
\end{enumerate}
\end{corollary}

\begin{remark}
Some necessary and sufficient conditions to ensure the constant sign of the periodic Green's function are given in \cite{cacidtv,cacid,HaklTor,torres1,zhanli,zhangMN05}. From the relations established in the previous corollary it is immediate to adapt all these results to the Green's functions of Neumann, Dirichlet and Mixed problems.
\end{remark}

\subsection{Comparison Principles} \label{sec-ppios-comp}
From the relations between the constant sign of the Green's functions (Corollary \ref{c-GP<0->GD<0}), as well as the explicit expressions obtained in Section 3 for such functions, we will deduce some comparison criteria  for Green's function of different problems and, as a consequence, for solutions of the related non homogeneous linear problems.

From (\ref{e-Green-DN-P}) and (\ref{e-N-D-P}), we obtain
\begin{corollary}
\label{c-Green-ND-p}
If $G_P[\tilde a,2\,T]\geq 0$ on $J\times J$, then
$$G_N[a,T](t,s)\geq |G_D[a,T](t,s)| \,(=-G_D[a,T](t,s))\qquad \forall \;(t,s)\in I\times I.$$
If $G_P[\tilde a,2\,T]< 0$ on $J\times J$, then
$$G_N[a,T](t,s) < G_D[a,T](t,s) \,(\le 0) \qquad \forall \; (t,s)\in I\times I.$$
\end{corollary}

As consequence, we deduce the following comparison principles:

\begin{theorem}
\label{t-comp-Neu-Per}
Suppose that $G_P[\tilde a,2\,T]\geq 0$ on $J\times J$. Let $u_N$ be the unique solution of problem \eqref{e-N} for $\sigma = \sigma_1$ and $u_D$ the unique solution of problem \eqref{e-D} for $\sigma = \sigma_2$. 

Suppose that $|\sigma_2(t)| \le \sigma_1(t)$ a.\,e. $t \in I$, then $|u_D(t)| \le u_N(t)$ for all $t \in I$.
\end{theorem}

\begin{proof}
By definition of Green's functions, as consequence of Corollary \ref{c-Green-ND-p} we have the following inequalities for every $t \in I$:
\begin{eqnarray*}
|u_D(t)|&=&\left|\int_0^T{G_D[a,T](t,s)\, \sigma_2(s)\,ds}\right| \le
\int_0^T{\left|G_D[a,T](t,s)\right|\,\left| \sigma_2(s)\right|\,ds} \\
&\le& \int_0^T{G_N[a,T](t,s)\,\left| \sigma_2(s)\right|\,ds} \le \int_0^T{G_N[a,T](t,s)\,\sigma_1(s)\,ds} = u_N(t).
\end{eqnarray*}
\end{proof}

\begin{theorem}
\label{t-comp-Neu-Dir-2}
Suppose that $G_P[\tilde a,2\,T]< 0$ on $J\times J$. Let $u_D$ be the unique solution of problem \eqref{e-D} for $\sigma = \sigma_1$ and $u_N$ the unique solution of problem \eqref{e-N} for $\sigma = \sigma_2$. 
\begin{enumerate}
\item If $0\le \sigma_2(t) \le \sigma_1(t)$ a.\,e. $t \in I$, then $u_N(t) \le u_D(t) \le 0$ for all $t \in I$.

\item If $0\ge \sigma_2(t) \ge \sigma_1(t)$ a.\,e. $t \in I$, then $u_N(t) \ge u_D(t) \ge 0$ for all $t \in I$.
\end{enumerate}
\end{theorem}

Analogously, from (\ref{e-Green-NM1-N}), (\ref{e-N-M1-N}), (\ref{e-Green-M2D-D}) and (\ref{e-M2-D-D}) we conclude that
\begin{corollary}
\label{c-Green-NM1-N-M2D-D}
If $G_N[\tilde a,2\,T]\geq 0$ on $J\times J$, then
$$G_N[a,T](t,s)\geq |G_{M_1}[a,T](t,s)|\, (=-G_{M_1}[a,T](t,s))\qquad \forall \;(t,s)\in I\times I.$$
If $G_N[\tilde a,2\,T]< 0$ on $J\times J$, then
$$G_N[a,T](t,s) < G_{M_1}[a,T](t,s) \, (\le 0)\qquad \forall \; (t,s)\in I\times I.$$
If $G_D[\tilde a,2\,T]\leq 0$ on $J\times J$, then
$$G_{M_2}[a,T](t,s) < G_{D}[a,T](t,s) \, (\le 0)\qquad \forall \; (t,s)\in I\times I.$$
\end{corollary}

As consequence, we deduce the following comparison principles

\begin{theorem}
\label{t-comp-Neu-M1}
Suppose that $G_N[\tilde a,2\,T]\geq 0$ on $J\times J$. Let $u_N$ be the unique solution of problem \eqref{e-N} for $\sigma = \sigma_1$ and $u_{M_1}$ the unique solution of \eqref{e-M-1} for $\sigma = \sigma_2$. 

Suppose that $|\sigma_2(t)| \le \sigma_1(t)$ a.\,e. $t \in I$, then $|u_{M_1}(t)| \le u_N(t)$ for all $t \in I$.
\end{theorem}

\begin{theorem}
\label{t-comp-Neu-M1-2}
Suppose that $G_N[\tilde a,2\,T]< 0$ on $J\times J$. Let $u_{M_1}$ be the unique solution of problem \eqref{e-M-1} for $\sigma = \sigma_1$ and $u_N$ the unique solution of \eqref{e-N} for $\sigma = \sigma_2$. 
\begin{enumerate}
\item If $0\le \sigma_2(t) \le \sigma_1(t)$ a.\,e. $t \in I$, then $u_N(t) \le u_{M_1}(t)\le 0$ for all $t \in I$.

\item If $0\ge \sigma_2(t) \ge \sigma_1(t)$ a.\,e. $t \in I$, then $u_N(t) \ge u_{M_1}(t)\ge 0$ for all $t \in I$.
\end{enumerate}
\end{theorem}

\begin{theorem}
\label{t-comp-D-M2}
Suppose that $G_D[\tilde a,2\,T]\le 0$ on $J\times J$. Let $u_D$ be the unique solution of problem \eqref{e-D} for $\sigma = \sigma_1$ and $u_{M_2}$ the unique solution of problem \eqref{e-M-2} for $\sigma = \sigma_2$. 
\begin{enumerate}
\item If $0\le \sigma_2(t) \le \sigma_1(t)$ a.\,e. $t \in I$, then $u_{M_2}(t) \le u_D(t)\le 0$ for all $t \in I$.

\item If $0\ge \sigma_2(t) \ge \sigma_1(t)$ a.\,e. $t \in I$, then $u_{M_2}(t) \ge u_D(t) \ge 0$ for all $t \in I$.
\end{enumerate}
\end{theorem}

Rewriting Corollaries \ref{c-Green-ND-p} and \ref{c-Green-NM1-N-M2D-D} in terms of eigenvalues, we have the following

\begin{corollary}
\label{c-Green-ND-p-2}
If $(\lambda_N(a,T)=)\,\lambda_P(\tilde a,2\,T)<0 \le \lambda_A(\tilde a,2\,T)$, then
$$G_N[a,T](t,s)\geq -G_D[a,T](t,s)\ge 0 \qquad \forall \;(t,s)\in I\times I.$$
If $(\lambda_N(a,T)=\lambda_N(\tilde{a},2\,T)\,=\lambda_P(\tilde{\tilde{a}},4\,T)=)\,\lambda_P(\tilde a,2\,T)>0$, then
$$G_N[a,T](t,s) < G_D[a,T](t,s)\le 0 \qquad \forall \; (t,s)\in I\times I$$
and
$$G_N[a,T](t,s) < G_{M_1}[a,T](t,s)\le 0 \qquad \forall \; (t,s)\in I\times I.$$
If $(\lambda_N(a,T)=\lambda_N(\tilde{a},2\,T)=)\,\lambda_P(\tilde{\tilde{a}},4\,T)<0 \le \lambda_A(\tilde{\tilde{a}},4\,T)$, then
$$G_N[a,T](t,s)\geq -G_{M_1}[a,T](t,s)\ge 0\qquad \forall \;(t,s)\in I\times I.$$
If $(\lambda_D(\tilde{a},2\,T)=)\,\lambda_{M_2}(a,T)>0$, then
$$G_{M_2}[a,T](t,s) < G_{D}[a,T](t,s)\le 0\qquad \forall \; (t,s)\in I\times I.$$
\end{corollary}

We also deduce the following result

\begin{corollary}
\label{c-Green-ND-p-3}
If $G_N[a,T]\geq 0$ on $I\times I$, then
\begin{enumerate}
\item $G_N[a,T](t,s)\leq 2\,G_P[\tilde a,2\,T](2\,T - t,s)$ on $I\times I$.
\item $0\ge G_D[a,T](t,s)\geq -2\,G_P[\tilde a,2\,T](2\,T - t,s)$ on $I\times I$.
\item $G_N[a,T](t,s)\leq 2\,G_N[\tilde a,2\,T](2\,T - t,s)$ on $I\times I$.
\item $0\ge G_{M_1}[a,T](t,s)\geq -2\,G_N[\tilde a,2\,T](2\,T - t,s)$ on $I\times I$.
\end{enumerate}
In particular, $G_P[\tilde a,2\,T](2\,T - t,s) \ge 0$ and $G_N[\tilde a,2\,T](2\,T - t,s) \ge 0$ on $I\times I$.
\end{corollary}
\begin{proof}
The inequalities are deduced from expressions (\ref{e-Green-DN-P}), (\ref{e-N-D-P}), (\ref{e-Green-NM1-N}) and (\ref{e-N-M1-N}) by taking into account that if $G_N[a,T]\ge 0$ on $I\times I$ then $G_D[a,T]\le 0$ and $G_{M_1}[a,T]\le 0$ on $I\times I$.
\end{proof}

\subsection{Global order of eigenvalues}
It can also be proved that there exists a certain order for the eigenvalues related to (\ref{e-N}), (\ref{e-D}), (\ref{e-M-1}) and (\ref{e-M-2}). To prove that, the following remark will be taken into account.

\begin{remark}
All the eigenvalues of Neumann, Dirichlet and Mixed problems are simple and the eigenfunction associated with the $k$-th eigenvalue (with $k=0,1,\dots$) takes the value zero exactly $k$ times on the open interval $(0,T)$. This general result is proved in \cite[Chapter 7]{weinberger} for Dirichlet and Mixed problems. For Neumann problem the result can be obtained from the previous ones by means of classical oscillation theory of Sturm-Liouville (\cite{simmons}).
\end{remark}

Consider then the following facts:
\begin{itemize}
\item[(i)] Let $\lambda_k^{N,T},\, \lambda_{k+1}^{N,T}\in\Lambda_N[a,T]$ be two consecutive eigenvalues of problem (\ref{e-N}) and let $u_k^{N,T}$ and $u_{k+1}^{N,T}$ be their associated eigenfunctions, with $k$ and $k+1$ zeros on the interval $[0,T]$, respectively.

If we consider the even extensions of $u_k^{N,T}$ and $u_{k+1}^{N,T}$ to the interval $[0,2\,T]$, it is clear that they have $2k$ and $2k+2$ zeros on $[0,2\,T]$, respectively, so there must exist an eigenvalue $\lambda\in\Lambda_N[\tilde{a},2\,T]$, $\lambda_k^{N,T}<\lambda<\lambda_{k+1}^{N,T}$, such that its associated eigenfunction has exactly $2k+1$ zeros on the interval $[0,2\,T]$. From the decomposition of the Neumann spectrum showed in Subsection \ref{s-mixed}, we have that, necessarily, $\lambda\in \Lambda_{M_1}[a,T]$.  

As we know that $\lambda_N(\tilde{a},2\,T)=\lambda_N(a,T)\equiv \lambda_0^{N,T}$ we conclude that
$$\dots < \lambda_k^{N,T} < \lambda_k^{M_1,T} < \lambda_{k+1}^{N,T} < \lambda_{k+1}^{M_1,T} < \dots$$

\item[(ii)] Analogously, we can easily see that $\Lambda_{M_2}[a,T]$ corresponds with eigenvalues of $\Lambda_D[\tilde{a},2\,T]$ whose eigenfunctions have an even number of zeros on $(0,2\,T)$ and $\Lambda_D[a,T]$ corresponds with eigenvalues of $\Lambda_D[\tilde{a},2\,T]$ whose eigenfunctions have an odd number of zeros on $(0,2\,T)$. Taking into account the fact that  $\lambda_D(\tilde{a},2\,T)=\lambda_{M_2}(a,T)\equiv \lambda_0^{M_2,T}$ we conclude that
$$\dots < \lambda_k^{M_2,T} < \lambda_k^{D,T} < \lambda_{k+1}^{M_2,T} < \lambda_{k+1}^{D,T} < \dots$$

\item[(iii)] Oscillation Theorem guarantees that the eigenvalues of periodic and anti-periodic problems related to the same interval (which we will denote as $\lambda_n$ and $\lambda'_n$, respectively) always appear in the following order
$$\lambda_0<\lambda'_1\le \lambda'_2 < \lambda_1 \le \lambda_2 <\lambda'_3 \le \lambda'_4 < \lambda_3 \le \lambda_4 \dots$$
\end{itemize}

Consequently, if we consider item $(iii)$ for problems (\ref{e-P-2T}) and (\ref{e-A-2T}) and we take into account the inequalities obtained in items $(i)$ and $(ii)$ we can affirm that
\begin{itemize}
\item In each pair $\{\lambda_{2k-1}, \ \lambda_{2k}\}$ of two consecutive eigenvalues of problem (\ref{e-P-2T}), one of them belongs to $\Lambda_{N}(a,T)$ and the other one belongs to $\Lambda_{D}(a,T)$.
\item In each pair $\{\lambda'_{2k-1}, \ \lambda'_{2k}\}$ of two consecutive eigenvalues of problem (\ref{e-A-2T}), one of them belongs to $\Lambda_{M_1}(a,T)$ and the other one belongs to $\Lambda_{M_2}(a,T)$.
\end{itemize}

The previous reasoning lets us conclude that the eigenvalues of problem (\ref{e-P-4T}) always appear in the following order:
$$\lambda_0^{N,T}< \left\{\lambda_0^{M_1,T}, \ \lambda_0^{M_2,T}\right\} < \left\{\lambda_0^{D,T}, \ \lambda_1^{N,T}\right\} <  \left\{\lambda_1^{M_1,T}, \ \lambda_1^{M_2,T}\right\} < \left\{\lambda_1^{D,T}, \ \lambda_2^{N,T}\right\}<\dots$$

As an immediate consequence we deduce

\begin{corollary}
The following properties hold for any $a\in L^1(I)$.
\begin{enumerate}
\item $\lambda_k^{N,T} < \lambda_k^{M_2,T} < \lambda_{k+1}^{N,T} < \lambda_{k+1}^{M_2,T}, \quad k=0,1,\dots$
\item $\lambda_k^{M_1,T} < \lambda_k^{D,T} < \lambda_{k+1}^{M_1,T} < \lambda_{k+1}^{D,T}, \quad k=0,1,\dots$
\end{enumerate}
\end{corollary}

\begin{remark}\label{obs-discrim}
In \cite[Chapter 1]{magnus} the following equalities are proved in the case of an even potential on $[0,2\,T]$:
\begin{eqnarray}
\label{e-y1-2T}
y_1(2\,T,\lambda)&=& 2\,y_1(T,\lambda)\,y'_2(T,\lambda)-1=1+2\,y'_1(T,\lambda)\,y_2(T,\lambda)\\
\label{e-y1-prima-2T}
y'_1(2\,T,\lambda)&=& 2\,y_1(T,\lambda)\,y'_1(T,\lambda)\\
\label{e-y2-2T}
y_2(2\,T,\lambda)&=& 2\,y_2(T,\lambda)\,y'_2(T,\lambda)\\
\label{e-y2-prima-2T}
y'_2(2\,T,\lambda)&=& y_1(2\,T,\lambda)
\end{eqnarray}
with $y_1$ and $y_2$ the fundamental solutions of equation \eqref{e-hill-intro}, defined in Theorem \ref{t-oscilacion}.

Furthermore, we deduce that, as $\tilde{a}$ is an even function, the decomposition of Neumann and Dirichlet spectrums in $2\,T$, 
$$\Lambda_N[\tilde{a},2\,T]=\Lambda_N[a,T]\cup \Lambda_{M_1}[a,T] \quad \text{and} \quad \Lambda_D[\tilde{a},2\,T]=\Lambda_D[a,T]\cup \Lambda_{M_2}[a,T],$$
is immediate from the equalities (\ref{e-y1-prima-2T}) and (\ref{e-y2-2T}). This deduction, despite being more direct than the one obtained in this work, does not give any information about the order of eigenvalues.

Moreover, the study developed in this work presents the additional interest of giving the exact expression of the Green's function of a problem from others. This let us deduce not only the decomposition of the corresponding spectrums, but also several comparison criteria both between the constant sign of the Green's function (Section \ref{sec-signo-cte}), and between such functions point by point and even between the solutions of the different associated linear problems (Section \ref{sec-ppios-comp}).
\end{remark}

\subsection{Examples}\label{sec-ejemplos}
We will see now some examples of the different situations that we could find.

\begin{example}\label{ej-cte}
If we consider the constant case $a(t)=0$, it is known that (see \cite{Cab}) 
$$\lambda_P(0,2\,T)=\lambda_N(0,T)=0 \quad 
\text{ and } \quad 
\lambda_A(0,2\,T)=\lambda_D(0,2\,T)=\left(\frac{\pi}{2\, T}\right)^2.$$

Moreover, denoting $\lambda=m^2>0$ and using \cite{CCiM} we obtain
$$
G_P[m^2,2\, T](t,s)= \frac{1}{2\,m\, \sin{m\, T}}
 \left\{
\begin{array}{ll}
\cos{(m \,(s-t+T))}, & \quad 0\leq s\leq t\leq 2 T \\ 
\vspace*{-0.5cm}\\
\cos{(m \,(s-t-T))},& \quad 0\leq t\leq s\leq 2 T
\end{array}
\right.
$$
and
$$
G_N[m^2, T](t,s)= \frac{1}{m\, \sin{m\, T}}
 \left\{
\begin{array}{ll}
\cos{(m \,s)} \, \cos{(m \,(T-t))}, & \quad 0\leq s\leq t\leq T \\
\vspace*{-0.5cm}\\
\cos{(m \,t)}\,  \cos{(m \,(T-s))}, & \quad 0\leq t\leq s\leq T
\end{array}
\right. .
$$

It is obvious that

$$G_N[m^2, T](0,0)=2\,G_P[m^2,2\, T](0,0)=\frac{1}{m\, \tan{m\, T}}.$$

As consequence we know that ${\lambda}_{M_1}(0,T)=\left(\frac{\pi}{2\, T}\right)^2$.

From the fact that
$$G_N[m^2, T](T,T)=2\,G_P[m^2,2\, T](T,T)=\frac{1}{m\, \tan{m\, T}},$$
we deduce that ${\lambda}_{M_2}(0,T)=\left(\frac{\pi}{2\, T}\right)^2$. This is also deduced from Corollary \ref{l-a(T-t)}.

We can use \cite{CCiM} to calculate the Green's functions for the different boundary conditions
$$
G_D[m^2, T](t,s)= \frac{1}{m\, \sin{m\, T}}
 \left\{
\begin{array}{ll}
\sin{(m\, s)} \, \sin{(m \,(t-T))}, & \quad 0\leq s\leq t\leq T \\
\vspace*{-0.5cm}\\
\sin{(m\, t)} \,  \sin{(m \,(s-T))}, & \quad 0\leq t\leq s\leq T
\end{array}
\right. ,
$$

$$
G_{M_1}[m^2, T](t,s)= \frac{1}{m\, \cos{m\, T}}
 \left\{
\begin{array}{ll}
\cos{(m\, s)} \, \sin{(m \,(t-T))}, &  \quad 0\leq s\leq t\leq T \\
\vspace*{-0.5cm}\\
\cos{(m\, t)} \, \sin{(m \,(s-T))}, & \quad 0\leq t\leq s\leq T
\end{array}
\right. ,
$$

$$
G_{M_2}[m^2, T](t,s)= \frac{-1}{m\, \cos{m\, T}}
 \left\{
\begin{array}{ll}
\sin{(m \,s)} \,\cos{(m \,(T-t))}, & \quad 0\leq s\leq t\leq T \\
\vspace*{-0.5cm}\\
\sin{(m \,t)} \,  \cos{(m \,(T-s))}, & \quad 0\leq t\leq s\leq T
\end{array}
\right. 
$$
and
$$
G_{A}[m^2, 2\,T](t,s)= \frac{-1}{2\,m\, \cos{m\, T}}
 \left\{
\begin{array}{ll}
\sin{(m\, (s-t+T))}, & \quad 0\leq s\leq t\leq T \\
\vspace*{-0.5cm}\\
\sin{(m \,(-s+t+T))}, & \quad 0\leq t\leq s\leq T
\end{array}
\right. .
$$

We observe then that $\lambda_D(0,T)=\left(\frac{\pi}{T}\right)^2$.

In this case $\Lambda_N[a,T]=\Lambda_D[a,T]=\Lambda_P[\tilde{a},2\,T]$ and $\Lambda_{M_1}[a,T]=\Lambda_{M_2}[a,T]=\Lambda_A[\tilde{a},2\,T]$.

Then, if we represent $\tilde{\Delta}(\lambda)=y_1(2\,T,\lambda)+y'_2(2\,T,\lambda)$ graphically, we obtain the following:
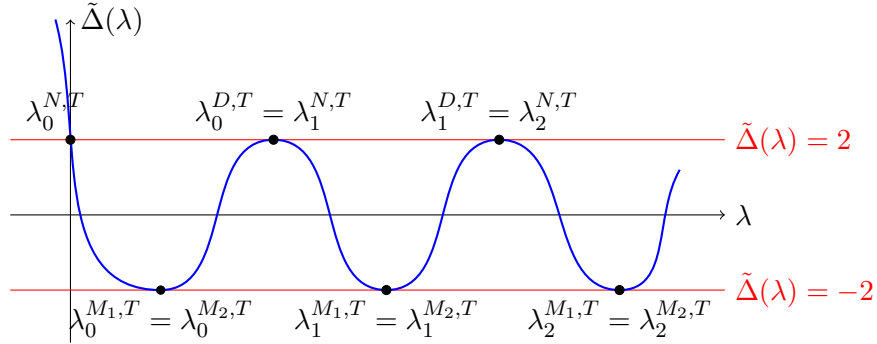
\begin{center}
\begin{tikzpicture}
\draw[->] (-2,0) -- (7.5,0) node[right] {$\lambda$};
\draw[->] (-1.2,-1.7) -- (-1.2,2.6) node[right] {$\tilde{\Delta}(\lambda)$};
\draw[red] (-2,1) -- (7.5,1) node[right] {$\tilde{\Delta}(\lambda)=2$};
\draw[red] (-2,-1) -- (7.5,-1) node[right] {$\tilde{\Delta}(\lambda)=-2$};
\draw [blue, thick] (-1.4,2.6) to [out=-75,in=180] (0,-1) to [out=0,in=180] (1.5,1) to [out=0,in=180] (3,-1) to [out=0,in=180] (4.5,1) to [out=0,in=180] (6.1,-1) to [out=0,in=240] (6.9,0.6);
\draw[fill] (-1.2,1) circle (0.06cm);
\draw (-1.4,1) node[above] {$\lambda_0^{N,T}$};
\draw[fill] (1.5,1) circle (0.06cm) node[above] {$\lambda_0^{D,T}=\lambda_1^{N,T}$};
\draw[fill] (4.5,1) circle (0.06cm) node[above] {$\lambda_1^{D,T}=\lambda_2^{N,T}$};
\draw[fill] (0,-1) circle (0.055cm) node[below] {$\lambda_0^{M_1,T}=\lambda_0^{M_2,T}$};
\draw[fill] (3,-1) circle (0.06cm) node[below] {$\lambda_1^{M_1,T}=\lambda_1^{M_2,T}$};
\draw[fill] (6.1,-1) circle (0.06cm) node[below] {$\lambda_2^{M_1,T}=\lambda_2^{M_2,T}$};
\end{tikzpicture}
\end{center}

\end{example}

\begin{example}\label{ej-cte-trozos}
If we consider $T=2$, and $a(t)=0$, if $t \in [0,1]$ and $a(t)=1/10$, for $t \in [1,2]$, the eigenfunctions can be directly obtained and we can verify that

$$\lambda_N(a,2)=\lambda_N(\tilde a,4)=\lambda_P(\tilde a,4)\approx -0.0508$$
$$\lambda_{M_2}(a,2)=\lambda_{D}(\tilde a,4)=\lambda_A(\tilde a,4)\approx 0.5346$$
$$\lambda_{M_1}(a,2)\approx 0.5984$$
$$\lambda_D(a,2)\approx 2.4170$$

Graphically, the situation would be the following
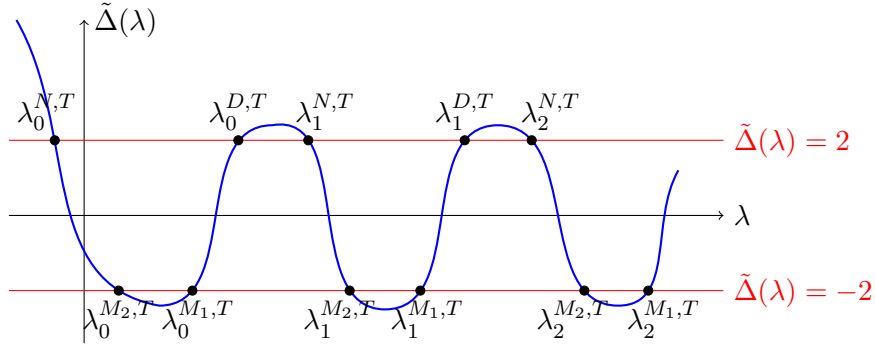
\begin{center}
\begin{tikzpicture}
\draw[->] (-2,0) -- (7.5,0) node[right] {$\lambda$};
\draw[->] (-1,-1.7) -- (-1,2.6) node[right] {$\tilde{\Delta}(\lambda)$};
\draw[red] (-2,1) -- (7.5,1) node[right] {$\tilde{\Delta}(\lambda)=2$};
\draw[red] (-2,-1) -- (7.5,-1) node[right] {$\tilde{\Delta}(\lambda)=-2$};
\draw [blue, thick] (-1.9,2.6) to [out=-60,in=170] (0,-1.2) to [out=0,in=180] (1.5,1.2) to [out=10,in=180] (3,-1.25) to [out=0,in=180] (4.5,1.2) to [out=0,in=180] (6.1,-1.2) to [out=0,in=240] (6.9,0.6);
\draw[fill] (-1.39,1) circle (0.06cm);
\draw (-1.5,1) node[above] {$\lambda_0^{N,T}$};
\draw[fill] (1.05,1) circle (0.06cm) node[above] {$\lambda_0^{D,T}$};
\draw[fill] (1.98,1) circle (0.06cm) node[above] {$ $};
\draw (2.2,1) node[above] {$\lambda_1^{N,T}$};
\draw[fill] (4.06,1) circle (0.06cm) node[above] {$\lambda_1^{D,T}$};
\draw[fill] (4.95,1) circle (0.06cm) node[above] {$ $};
\draw (5.2,1) node[above] {$\lambda_2^{N,T}$};
\draw[fill] (-0.54,-1) circle (0.06cm) node[below] {$\lambda_0^{M_2,T}$};
\draw[fill] (0.44,-1) circle (0.06cm) node[below] {$ $};
\draw (0.51,-1) node[below] {$\lambda_0^{M_1,T}$};
\draw[fill] (2.53,-1) circle (0.06cm) node[below] {$ $};
\draw (2.4,-1) node[below] {$\lambda_1^{M_2,T}$};
\draw[fill] (3.47,-1) circle (0.06cm) node[below] {$ $};
\draw (3.56,-1) node[below] {$\lambda_1^{M_1,T}$};
\draw[fill] (5.65,-1) circle (0.06cm) node[below] {$ $};
\draw (5.5,-1) node[below] {$\lambda_2^{M_2,T}$};
\draw[fill] (6.5,-1) circle (0.06cm) node[below] {$ $};
\draw (6.7,-1) node[below] {$\lambda_2^{M_1,T}$};
\end{tikzpicture}
\end{center}

Note that the eigenvalue of problem (\ref{e-M-2}) always appears before the one of problem (\ref{e-M-1}). In addition, the order between the eigenvalues of (\ref{e-N}) and (\ref{e-D}) is also maintained.
\end{example}

\begin{example}\label{ej-cost}
Considering $T=\pi$ and $a(t)=\cos t$, we obtain the following approximations

$$\lambda_N(a,\pi)=\lambda_N(\tilde a,2\,\pi)=\lambda_P(\tilde a,2\,\pi)=\lambda_P(\tilde{\tilde{a}},4\,\pi)\approx -0.378$$
$$\lambda_{M_1}(a,\pi)=\lambda_A(\tilde{a},2\,\pi)\approx -0.348$$
$$\lambda_{M_2}(a,\pi)=\lambda_D(\tilde{a},2\,\pi)\approx 0.5948$$
$$\lambda_{D}(a,\pi)\approx 0.918$$

Graphically we would observe the following
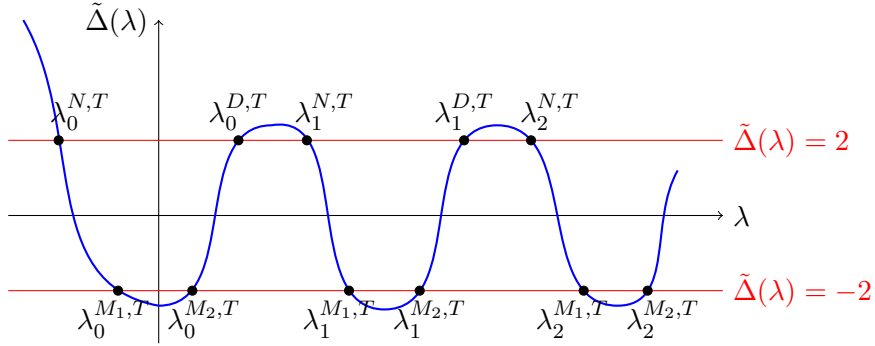
\begin{center}
\begin{tikzpicture}
\draw[->] (-2,0) -- (7.5,0) node[right] {$\lambda$};
\draw[->] (0,-1.7) -- (0,2.6) node[left] {$\tilde{\Delta}(\lambda)$};
\draw[red] (-2,1) -- (7.5,1) node[right] {$\tilde{\Delta}(\lambda)=2$};
\draw[red] (-2,-1) -- (7.5,-1) node[right] {$\tilde{\Delta}(\lambda)=-2$};
\draw [blue, thick] (-1.8,2.6) to [out=-60,in=170] (0,-1.2) to [out=0,in=180] (1.5,1.2) to [out=10,in=180] (3,-1.25) to [out=0,in=180] (4.5,1.2) to [out=0,in=180] (6.1,-1.2) to [out=0,in=240] (6.9,0.6);
\draw[fill] (-1.33,1) circle (0.06cm) node[above] {$ $};
\draw (-1.05,1) node[above] {$\lambda_0^{N,T}$};
\draw[fill] (1.06,1) circle (0.06cm) node[above] {$\lambda_0^{D,T}$};
\draw[fill] (1.97,1) circle (0.06cm) node[above] {$ $};
\draw (2.2,1) node[above] {$\lambda_1^{N,T}$};
\draw[fill] (4.06,1) circle (0.06cm) node[above] {$\lambda_1^{D,T}$};
\draw[fill] (4.95,1) circle (0.06cm) node[above] {$ $};
\draw (5.2,1) node[above] {$\lambda_2^{N,T}$};
\draw[fill] (-0.54,-1) circle (0.06cm) node[below] {$ $};
\draw (-0.6,-1) node[below] {$\lambda_0^{M_1,T}$};
\draw[fill] (0.447,-1) circle (0.06cm) node[below] {$ $};
\draw (0.6,-1) node[below] {$\lambda_0^{M_2,T}$};
\draw[fill] (2.53,-1) circle (0.06cm) node[below] {$ $};
\draw (2.4,-1) node[below] {$\lambda_1^{M_1,T}$};
\draw[fill] (3.47,-1) circle (0.06cm) node[below] {$ $};
\draw (3.6,-1) node[below] {$\lambda_1^{M_2,T}$};
\draw[fill] (5.65,-1) circle (0.06cm) node[below] {$ $};
\draw (5.5,-1) node[below] {$\lambda_2^{M_1,T}$};
\draw[fill] (6.5,-1) circle (0.06cm) node[below] {$ $};
\draw (6.7,-1) node[below] {$\lambda_2^{M_2,T}$};
\end{tikzpicture}
\end{center}

In this case, the eigenvalue of (\ref{e-M-1}) is smaller than the one of (\ref{e-M-2}). Again, the order between the eigenvalues of (\ref{e-N}) and (\ref{e-D}) is maintained.
\end{example}

The following example shows that eigenvalues related to problem (\ref{e-N}) do not necessarily have to alternate with the ones related to (\ref{e-D}).

\begin{example}\label{ej-cos2t}
Considering $T=\pi$ and $a(t)=\cos{2\,t}$, we obtain the following approximation for the spectrum of the considered problems
$$\Lambda_P[\tilde{\tilde{a}},4\,T]=\{-0.1218,\,0.0923,\,0.47065,\,1.4668,\,2.34076,\,3.9792,4.1009,\dots\}$$
$$\Lambda_P[\tilde{a},2\,T]=\{-0.1218,\,0.47065,\,1.4668,\,3.9792,\,4.1009,\dots\}$$
$$\Lambda_{M_1}[a,T]=\Lambda_{M_2}[a,T]=\Lambda_A[\tilde{a},2\,T]=\{0.0923,\,2.34076,\dots\}$$
$$\Lambda_N[a,T]=\{-0.1218,\,0.47065,\,4.1009,\dots\}$$
$$\Lambda_D[a,T]=\{1.4668,\,3.9792,\dots\}$$

We observe that in this case
$$\lambda_0^{N,T}<\lambda_1^{N,T}<\lambda_0^{D,T}<\lambda_1^{D,T}<\lambda_2^{N,T}.$$

Note that the eigenvalues of Mixed problems are the same. This is due to the fact that ${a(t)=a(T-t)}$. Consequently, all the eigenvalues of $\Lambda_A[\tilde{a},2\,T]$ are a double root of ${\tilde{\Delta}(\lambda)=-2}$.

Graphically,
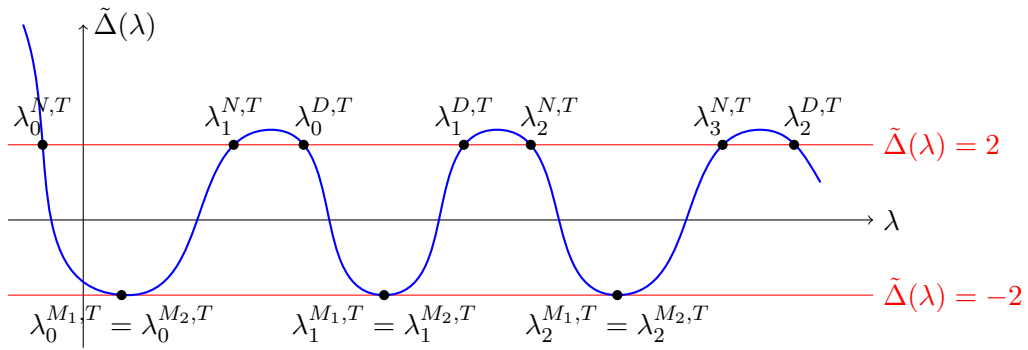
\begin{center}
\begin{tikzpicture}
\draw[->] (-2,0) -- (9.5,0) node[right] {$\lambda$};
\draw[->] (-1,-1.7) -- (-1,2.6) node[right] {$\tilde{\Delta}(\lambda)$};
\draw[red] (-2,1) -- (9.5,1) node[right] {$\tilde{\Delta}(\lambda)=2$};
\draw[red] (-2,-1) -- (9.5,-1) node[right] {$\tilde{\Delta}(\lambda)=-2$};
\draw [blue, thick] (-1.8,2.6) to [out=-70,in=180] (-0.4,-1) to [out=0,in=180] (1.5,1.2) to [out=0,in=180] (3,-1) to [out=0,in=180] (4.5,1.2) to [out=0,in=180] (6.1,-1) to [out=0,in=180] (8,1.2) to [out=0, in=120] (8.8,0.5);
\draw[fill] (-1.54,1) circle (0.06cm) node[above] {$\lambda_0^{N,T}$};
\draw[fill] (1,1) circle (0.06cm) node[above] {$\lambda_1^{N,T}$};
\draw[fill] (1.93,1) circle (0.06cm) node[above] {$ $};
\draw (2.2,1) node[above] {$\lambda_0^{D,T}$};
\draw[fill] (4.06,1) circle (0.06cm) node[above] {$\lambda_1^{D,T}$};
\draw[fill] (4.95,1) circle (0.06cm) node[above] {$ $};
\draw (5.2,1) node[above] {$\lambda_2^{N,T}$};
\draw[fill] (7.5,1) circle (0.06cm) node[above] {$\lambda_3^{N,T}$};
\draw[fill] (8.45,1) circle (0.06cm) node[above] {$ $};
\draw (8.7,1) node[above] {$\lambda_2^{D,T}$};
\draw[fill] (-0.49,-1) circle (0.06cm) node[below] {$\lambda_0^{M_1,T}=\lambda_0^{M_2,T}$};
\draw[fill] (3,-1) circle (0.06cm) node[below] {$\lambda_1^{M_1,T}=\lambda_1^{M_2,T}$};
\draw[fill] (6.1,-1) circle (0.06cm) node[below] {$\lambda_2^{M_1,T}=\lambda_2^{M_2,T}$};
\end{tikzpicture}
\end{center}
\end{example}

\begin{remark}
The numerical results obtained in the considered examples suggest an order of eigenvalues even more precise than the one theoretically proved in this paper.

It is observed that the eigenvalues of mixed problems alternate, with one eigenvalue of a mixed problem between two consecutive eigenvalues of the other one, and reciprocally. This has been observed in all the considered examples in which the spectrums of the two mixed problems are different (Examples  \ref{ej-cte-trozos} and \ref{ej-cost}), independently of which of them appears first.

We also appreciate in the examples an alternation between Neumann and Dirichlet eigenvalues except for the case in which the spectrum of the mixed problems is the same (in this case the order of appearance of Dirichlet and Neumann changes from one pair of eigenvalues and the next one, as we can see in Example \ref{ej-cos2t}).

This situation suggests the existence of some property justifying this fact. However, for the moment it has not been formally proved and these speculations are uniquely based on the numerical results obtained while working with different potentials.
\end{remark}

\end{document}